\documentclass{amsart}
\usepackage{xcolor}
\usepackage{quiver}
\usepackage{comment}
\usepackage{float}
\usepackage{hyperref}
\usepackage{cleveref}
\usepackage{algorithm2e}
\usepackage{graphicx} 
\title{Products in spin${}^c$-cobordism}
\author{Hassan Abdallah}
\author{Andrew Salch}
\address{Department of Mathematics, Wayne State University, Detroit, MI, USA}
\email{hassan@wayne.edu, asalch@wayne.edu}
\newtheorem{theorem}{Theorem}[section]
\newcounter{lettered}
\newtheorem{mainthm}[lettered]{Theorem}
\newtheorem{maincor}[lettered]{Corollary}

\newtheorem{definition}[theorem]{Definition}
\newtheorem{proposition}[theorem]{Proposition}
\newtheorem{corollary}[theorem]{Corollary} 
\theoremstyle{definition}

\newtheorem{remark}[theorem]{Remark}
\DeclareMathOperator{\Sq}{{\rm Sq}}
\DeclareMathOperator{\Spec}{{\rm Spec}}

\DeclareMathOperator{\im}{{\rm im\:}}
\DeclareMathOperator{\Subsets}{{\mathcal{S}\mbox{\em{ub}}}}
\DeclareMathOperator{\Ext}{{\rm Ext}}

\begin{document}
\begin{abstract}
We calculate the mod $2$ spin$^c$-cobordism ring up to uniform $F$-isomorphism (i.e., inseparable isogeny). As a consequence we get the prime ideal spectrum of the mod $2$ spin$^c$-cobordism ring. We also calculate the mod $2$ spin$^c$-cobordism ring ``on the nose'' in degrees $\leq 33$. We construct an infinitely generated nonunital subring of the $2$-torsion in the spin$^c$-cobordism ring. We use our calculations of product structure in the spin and spin$^c$ cobordism rings to give an explicit example, up to cobordism, of a compact $24$-dimensional spin manifold which is not cobordant to a sum of squares, which was asked about in a 1965 question of Milnor.

\end{abstract}

\maketitle

\section{Introduction and summary of results}

\subsection{Spin$^c$ cobordism}
\label{spinc cobordism intro subsection}
A spin${}^c$-structure on a compact smooth $n$-dimensional manifold $M$ is a reduction of its structure group from $O(n)$ to $Spin^c(n)$. We find the following perspective illuminating: a compact smooth manifold is
\begin{itemize} 
\item orientable if its first Stiefel--Whitney class $w_1$ vanishes
\item and admits a spin structure if its first two Stiefel--Whitney classes, $w_1$ and $w_2$, both vanish.
\end{itemize}
A spin${}^c$-structure is intermediate between an orientation and a spin structure. Specifically, a compact smooth manifold $M$ admits a spin${}^c$ structure if its first Stiefel--Whitney class $w_1$ vanishes, and its second Stiefel--Whitney class $w_2$ is a reduction of an integral class. That is, $w_2\in H^2(M;\mathbb{F}_2)$ is in the image of the reduction-of-coefficients map $H^2(M;\mathbb{Z}) \rightarrow H^2(M;\mathbb{F}_2)$. For these and many other relevant facts, consult Stong's book \cite{MR0248858}.

The spin${}^c$-cobordism ring, written $\Omega^{Spin^c}_*$, is the ring of spin${}^c$-cobordism classes of compact smooth spin${}^c$-manifolds. The addition is given by disjoint union of manifolds, while the multiplication is Cartesian product. There are several reasons to care about spin${}^c$-cobordism: aside from its applications to mathematical physics, e.g. \cite{Blumenhagen_Cribiori_Kneißl_Makridou_2023} and \cite{ertem2020weyl}, spin${}^c$-cobordism is of particular interest because it is one of the {\em complex-oriented} cobordism theories, and consequently there exists a one-dimensional group law on $\Omega^{Spin^c}_*$ which describes how the first Chern class in spin${}^c$-cobordism behaves on a tensor product of complex line bundles. See \cite{MR1324104} or \cite{coctalos} for these classical ideas, whose consequences for complex cobordism (as in \cite{MR860042}) have been enormous, but whose consequences for spin${}^c$-cobordism have apparently never been fully explored.

Since spin${}^c$-cobordism is an example of a ``$(B,f)$-cobordism theory'' in the sense of Thom, the general results of \cite{MR0061823} ensure that there exists a spectrum $MSpin^c$ such that $\pi_*(MSpin^c)\cong \Omega^{Spin^c}_*$. The homotopy type of the spectrum $MSpin^c$ is understood as follows. 
\begin{description}
\item[Away from $2$] 
The map $\pi: BSpin^{c} \longrightarrow BSO \times K(\mathbb{Z},2)$ is an odd-primary homotopy equivalence and induces an isomorphism $\Omega^{Spin^c}_*[\frac{1}{2}] \cong \Omega^{SO}_*(K(\mathbb{Z},2))[\frac{1}{2}]$, and consequently $MSpin^c[\frac{1}{2}] \cong MSO[\frac{1}{2}] \wedge \mathbb{C}P^{\infty}$. 
\item[At $2$]
In 1966, Anderson, Brown, and Peterson \cite{MR0190939},\cite{MR0219077} proved that $MSpin^c$ splits $2$-locally as a wedge of suspensions of the connective complex $K$-theory spectrum $ku$ and the mod $2$ Eilenberg-Mac Lane spectrum $H\mathbb{F}_2$:
\begin{align}
\label{abp splitting} MSpin^c_{(2)} &\simeq Z \vee \coprod_{J} \Sigma^{4\left| J\right|} ku_{(2)}
\end{align}
where the coproduct (i.e., wedge sum) is taken over all partitions (i.e., unordered finite tuples of positive integers) $J$, and $\left| J\right|$ denotes the sum of the entries of $J$. 

Not much is known about the summand $Z$ in \eqref{abp splitting}, other than that
\begin{itemize}
\item it is a coproduct of suspensions of copies of $H\mathbb{F}_2$,
\item and from a Poincar\'{e} series \cite{MR0190939}, it is known how to solve inductively for the number of copies of $\Sigma^n H\mathbb{F}_2$ in $Z$, for each $n$.
\end{itemize}
In that sense, $Z$ is understood {\em additively}. 
\end{description}

This purely additive understanding of $Z$, and consequently of $2$-local $\Omega^{Spin^c}_*$, is not entirely satisfying.  To see the problem, consider the following table, which we reproduce from Bahri--Gilkey \cite{MR0883375}:
\begin{table}[H]
\begin{tabular}{||c|c||c|c||c|c||c|c||}
\hline
 $n$ & $\dim_{\mathbb{F}_2}\pi_nZ$ &
 $n$ & $\dim_{\mathbb{F}_2}\pi_nZ$ &
 $n$ & $\dim_{\mathbb{F}_2}\pi_nZ$ &
 $n$ & $\dim_{\mathbb{F}_2}\pi_nZ$  \\
 \hline
 0 & 0 & 
 8 & 0& 
 16 & 0& 
 24 & 2 \\
 1 & 0 & 
 9 & 0& 
 17 & 0& 
 25 & 0 \\
 2 & 0 & 
 10 & 1& 
 18 & 3& 
 26 & 9 \\
 3 & 0& 
 11 & 0& 
 19 & 0& 
 27 & 0 \\
 4 & 0& 
 12 & 0& 
 20 & 1& 
 28 & 4 \\
 5 & 0& 
 13 & 0& 
 21 & 0& 
 29 & 1 \\
 6 & 0& 
 14 & 1& 
 22 & 5& 
 30 & 14 \\
 7 & 0& 
 15 & 0& 
 23 & 0& 
 31 & 1 \\
\hline
\end{tabular}
\caption{}\label{bahri-gilkey table 1}
\end{table}
The $\mathbb{F}_2$-linear dimension of $\pi_nZ$, as recorded in \cref{bahri-gilkey table 1}, is equivalently the number of copies of $\Sigma^n H\mathbb{F}_2$ in $2$-local $MSpin^c$, and equivalently the $\mathbb{F}_2$-rank of the $2$-torsion subgroup of $\Omega^{Spin^c}_n$. Hence this table is telling us about the $2$-torsion in the spin${}^c$-cobordism ring. 
One has the sense that some deep pattern is present in the distribution of the $2$-torsion, but whatever it is, it cannot be seen clearly from these $\mathbb{F}_2$-ranks, nor from the Poincar\'{e} series used to inductively compute them.

However, since $\pi_*(Z)$ is precisely the $2$-torsion in $\Omega^{Spin^c}_*$, $\pi_*(Z)$ is not only a summand but also an {\em ideal} in $\Omega^{Spin^c}_*$.
One wants to understand $\pi_*(Z)$ {\em multiplicatively,} i.e., one wants to be able to describe the ring structure on $\Omega^{Spin^c}_*$, including its $2$-torsion elements. A reasonably clear description of $\Omega^{Spin^c}_*$ as a ring would yield a far more illuminating understanding of $\pi_*(Z)$ than the inductive formula for its $\mathbb{F}_2$-rank in each degree, which is presently all we have. 

Fifty years after the additive structure of $MSpin^c$ was described by Anderson--Brown--Peterson, the problem of calculating the ring structure of $\Omega^{Spin^c}_*$ remains open. The purpose of this paper is to make progress towards a solution to this problem, restricting to the $2$-local case, which is the most difficult.

\begin{remark}
In principle, the ring structure on $\Omega^{Spin^c}_*$ {\em away} from $2$ is understood, although only in a rather indirect way. Here is how it works: from the complex-orientability of $MSO$, one gets an isomorphism of rings $MSO[\frac{1}{2}]^*(\mathbb{C}P^{\infty})\cong MSO[\frac{1}{2}]^*[[X]]$. The ring $MSO[\frac{1}{2}]^*[[X]]$ is also the ``covariant bialgebra'' of the formal group law of $MSO[\frac{1}{2}]_*$, in the sense of \cite[Chapter 36]{MR2987372}. Hence one can use the formal group law on $MSO[\frac{1}{2}]^*$ (whose universal property is given by \cite{baker2014msp}) to understand the coproduct on $MSO[\frac{1}{2}]^*[[X]]$, whose dual, in an appropriate sense, is responsible for the ring structure on $MSO[\frac{1}{2}]_*(\mathbb{C}P^{\infty}) \cong \Omega^{Spin^c}_*[\frac{1}{2}]$. This gives a means of understanding the ring $\Omega^{Spin^c}_*[\frac{1}{2}]$, although we know of nowhere in the literature where this has been carried out in any further detail.
\end{remark}

\subsection{The mod $2$ spin${}^c$-cobordism ring in low degrees}
\label{spinc-cobordism ring in low degrees}
There are natural forgetful maps, $\Omega^{Spin}_* \rightarrow \Omega^O_*$ and $\Omega^{Spin^c}_* \rightarrow \Omega^O_*$, from the spin and spin${}^c$-cobordism rings to the unoriented cobordism ring. Since $\Omega^O_*$ is well-understood, it is natural to try to understand the images of these forgetful maps.
In the 1968 book \cite[pg. 351]{MR0248858}, Stong asks:
\begin{quotation}
\underline{Open question:} Can one determine these images nicely as subrings of $\Omega^O_*$?
\end{quotation}
Our approach to understanding the mod $2$ spin${}^c$-cobordism ring begins by answering Stong's open question in a range of degrees. 
We use the Anderson--Brown--Peterson splitting \cite{MR0190939}, product structure in the Adams spectral sequences, and Thom's determination of $\Omega^O_*$ using symmetric polynomials \cite{MR0061823} to develop a method for calculating the image of the map $\Omega^{Spin^c}_*\rightarrow \Omega^O_*$ through degree $d$, for any fixed choice of integer $d$. Our method gives a presentation for $\Omega^{Spin^{c}}/(2,\beta)$ through degree $d$, since the map $\Omega^{Spin^{c}}/(2,\beta) \longrightarrow \Omega^O_*$ is injective. We carry out computer calculation using our method to obtain our first main theorem:
\begin{mainthm}[Theorem \ref{mspinc_subring in main text}]\label{mspinc_subring} 
The subring of the mod $2$ spin${}^c$-cobordism ring $\Omega^{Spin^{c}}_{*}\otimes_{\mathbb{Z}}\mathbb{F}_2$ generated by all homogeneous elements of degree $\leq 33$ is isomorphic to
\begin{align*}
\mathbb{F}_{2}[\beta, Z_4,Z_8,Z_{10},Z_{12},Z_{16},Z_{18},Z_{20},Z_{22},Z_{24},Z_{26}, Z_{28},Z_{32}, \\ T_{24}, T_{29}, T_{31}, T_{32}, T_{33}]/I
\end{align*}
where $I$ is the ideal generated by the relations:
\begin{itemize}
\item 
$\beta Z_{i}=0$ for each $i\equiv 2 \mod 4$,
\item and 
$\ \beta T_{i}=0$ and $T_i^2 = U_{2i}$ for $i\in \{24,29,31,32,33\}$, where each $U_i$ is a particular polynomial in the generators $Z_{j}$ with $j\leq i-20$. The polynomial $U_i$ is described explicitly preceding Theorem \ref{mspinc mod 2 and beta}.
\end{itemize}
The degrees of the generators are as follows: $\beta = [\mathbb{C}P^1]$ is in degree $2$, while $Z_i$ and $T_i$ are each in degree $i$.
\end{mainthm}

With Theorem \ref{mspinc_subring} in hand, the patterns in \cref{bahri-gilkey table 1} become completely clear: in each degree in this range, one can see {\em why} the $\mathbb{F}_2$-linear dimension of the $2$-torsion subgroup of $\Omega^{Spin^c}_*$ takes the particular value it takes, as follows.
Since Anderson--Brown--Peterson proved that the $2$-torsion coincides with the $\beta$-torsion in $\Omega^{Spin^c}_*$, in degrees $n\leq 33$ the $2$-torsion in $\Omega^{Spin^c}_n$ is simply the $\mathbb{F}_2$-linear combinations of the monomials in the generators $Z_i,T_i$ such that at least one of the factors is $\beta$-torsion, i.e., at least one of the factors is either a generator $Z_i$ with $i\equiv 2\mod 4$, or a generator $T_i$. 
Here is the same table as \cref{bahri-gilkey table 1}, but augmented with an $\mathbb{F}_2$-linear basis in each degree, using the multiplicative structure from Theorem \ref{mspinc_subring}. We start in degree $10$ since there is no nontrivial $2$-torsion in $\Omega^{Spin^c}_*$ below degree $10$.
\begin{table}[H]
\begin{tabular}{||c|c|c||}
\hline
 $n$ & $\dim_{\mathbb{F}_2}\pi_nZ$ & $\mathbb{F}_2$-linear basis for $\pi_nZ$ \\
 \hline
 10 & 1& $Z_{10}$ \\
 11,12,13 & 0 & \\
 14 & 1& $Z_4Z_{10}$ \\
 15, 16, 17 & 0 & \\
 18 & 3& $Z_4^2Z_{10},  Z_8Z_{10}, Z_{18}$ \\
 19 & 0 & \\
 20 & 1 & $Z_{10}^2$ \\
 21 & 0 & \\
 22 & 5 & $Z_4^3Z_{10}, Z_4Z_8Z_{10}, Z_{12}Z_{10}, Z_4Z_{18}, Z_{22}$ \\
 23 & 0 & \\
 24 & 2 & $Z_4Z_{10}^2, T_{24}$ \\
 25 & 0 & \\
 26 & 9 & $Z_4^4Z_{10}, Z_4^2Z_8Z_{10}, Z_8^2Z_{10}, Z_4Z_{12}Z_{10}, Z_{16}Z_{10},$ \\ & & $Z_4^2Z_{18}, Z_8Z_{18}, Z_4Z_{22}, Z_{26}$ \\
 27 & 0 & \\
 28 & 4 & $Z_4^2Z_{10}^2, Z_8Z_{10}^2, Z_{10}Z_{18},Z_4T_{24}$ \\
 29 & 1 & $T_{29}$ \\
 30 & 14 & $Z_4^5Z_{10}, Z_4^3Z_8Z_{10},Z_4Z_8^2Z_{10}, Z_4^2Z_{12}Z_{10}, Z_8Z_{12}Z_{10},$\\ && $Z_4Z_{16}Z_{10}, Z_{20}Z_{10}, Z_4^3Z_{18},Z_4Z_8Z_{18}, Z_{12}Z_{18},$ \\ && $Z_4^2Z_{22}, Z_8Z_{22}, Z_4Z_{26}, Z_{10}^3$ \\
 31 & 1 & $T_{31}$ \\
 32 & 8 & $Z_4^3Z_{10}^2,Z_4Z_8Z_{10}^2,Z_{12}Z_{10}^2, Z_4Z_{10}Z_{18}, Z_{10}Z_{22},$\\ && $Z_4^2T_{24}, Z_8T_{24}, T_{32}$ \\
 33 & 2 & $Z_4T_{29}, T_{33}$ \\
\hline 

\end{tabular}
\caption{}\label{bahri-gilkey table 2}
\end{table}
One can also read off the product structure on the $2$-torsion in $\Omega^{Spin^c}_*$ in degrees $\leq 33$ from this table, since it is given by multiplication of monomials along with the relations from Theorem \ref{mspinc_subring}.

It is evident from Theorem \ref{mspinc_subring} that, in degrees $\leq 33$, $\Omega^{Spin^c}_*$ has a subring generated by elements $Z_4, Z_8, Z_{12},Z_{16}, \dots$ and by elements $Z_{2i}$ with $i$ odd and not one less than a power of $2$, subject to the relations $2 Z_{2i} = 0 = \beta Z_{2i}$ for all odd $i$. We are able to show that this pattern extends into all degrees, and goes some way to describing the ideal $\pi_*(Z)$ of $2$-torsion elements of $\Omega^{Spin^c}_*$ in multiplicative terms:
\begin{mainthm}[Theorem \ref{large nonunital subring thm}]\label{large nonunital...in intro} 
Consider the spin${}^c$-cobordism ring as a graded algebra over the graded ring $S:= \mathbb{Z}_{(2)}[\beta, Z_{2j} : j+1 \mbox{ not\ a\ power\ of\ } 2]/(\beta Z_{2j},2Z_{2j} \mbox{\ for\ odd\ } j).$
Let $J$ be the ideal of $S$ generated by all the elements $Z_{2j}$ with $j$ odd.
Then $J$ embeds, as a non-unital graded $S$-algebra, into the $2$-torsion ideal $\pi_*(Z)$ of the spin${}^c$-cobordism ring.
\end{mainthm}
Theorem \ref{large nonunital...in intro} describes the multiplicative structure of some, but not all, of the $2$-torsion in $\Omega^{Spin^c}_*$. For example, in degrees $\leq 33$, it accounts for precisely those monomials in \cref{bahri-gilkey table 2} which are {\em not} divisible by the elements $T_i$. In particular, the lowest-degree $2$-torsion element of $\Omega^{Spin^c}_*$ which is not described by Theorem \ref{large nonunital subring thm} is $T_{24}\in \Omega^{Spin^c}_{24}$. 

We have emphasized that the structure of the spin${}^c$-cobordism ring has remained mysterious, despite the great control over its additive structure which comes from the Anderson--Brown--Peterson splitting. In these respects, spin${}^c$-bordism is not unique: the structure of the spin-bordism ring is also quite mysterious, despite Anderson--Brown--Peterson \cite{MR0190939},\cite{MR0219077} also constructing a splitting of the $2$-local spin bordism spectrum $MSpin$ into copies of $H\mathbb{F}_2$, $ko_{(2)}$, and a suitably-connected cover of $ko_{(2)}$. Our methods were developed for spin${}^c$-cobordism, but in this paper we also apply them to spin cobordism. In \Cref{mspin_image}, we obtain a calculation of the image of the map $\Omega^{Spin}_*\rightarrow \Omega^O_*$ through degree 31. This has a noteworthy geometric consequence, which we will now describe.

\subsection{Milnor's $24$-dimensional spin manifold.}
In the 1965 paper \cite{MR0180977}, Milnor asks this question:
\begin{quotation}
 Problem. Does there exist a spin manifold $\Sigma$ of dimension $24$ so that $s_6(p_1, \dots ,p_6)[\Sigma] \equiv 1 \mod 2$?
\end{quotation}
Here $s_6$ is a certain symmetric polynomial (the sixth Girard--Newton polynomial, which converts from the elementary symmetric polynomials to the power sum polynomials), and $p_1,\dots ,p_6$ are Pontryagin classes. 
The reason for Milnor's question is that, in \cite{MR0180977}, Milnor proves that, for a compact smooth manifold $M$ of dimension $\leq 23$, the following conditions are equivalent:
\begin{enumerate}
\item $M$ is unorientedly cobordant to a spin manifold.
\item The Stiefel--Whitney numbers of $M$ involving $w_1$ and $w_2$ are all zero.
\item $M$ is unorientedly cobordant to $N\times N$, with $N$ an orientable compact manifold.
\end{enumerate}
Milnor points out that, if there exists a compact spin manifold $\Sigma$ whose Pontryagin number $s_6(p_1, \dots ,p_6)[\Sigma]$ is odd, then these conditions would fail to be equivalent in dimension $24$. Anderson--Brown--Peterson \cite{MR0190939},\cite{MR0219077} established that, as a consequence of their splitting of $2$-local $MSpin$, there does indeed exist such a compact spin manifold $\Sigma$. However, it seems that no explicit description of that $24$-dimensional compact spin manifold has been given in the literature (or anywhere else, as far as we know).

In Theorem \ref{milnor mfld thm}, we give an explicit formula for the unoriented bordism class of such a compact spin manifold $\Sigma$, as a disjoint union of products of real projective spaces and squares of Dold manifolds. We refer the reader to the Theorem \ref{milnor mfld thm} for a statement of that formula, which is lengthy. The formula is obtained using our calculation of the image of the map $\Omega^{Spin}_{24} \rightarrow \Omega^O_{24}$ and the manifold representatives calculated in \Cref{thom_manifolds}.  

\subsection{Determination of the mod $2$ spin${}^c$-cobordism ring up to inseparable isogeny}
Thom's famous calculation \cite{MR0061823} established that the unoriented bordism ring $\Omega^O_*$ is isomorphic to a polynomial algebra over $\mathbb{F}_2$. 
A theorem of Stong \cite[Proposition 14]{MR192516} shows that the spin${}^c$ cobordism ring, reduced modulo torsion and then reduced modulo $2$, is also isomorphic to a polynomial $\mathbb{F}_2$-algebra. 

By contrast, the spin$^{c}$ cobordism ring cannot itself be isomorphic to a polynomial algebra, since by \cite{MR0190939}, it has $2$-torsion but is not an $\mathbb{F}_2$-algebra, hence it has nontrivial zero divisors. Similarly, since the mod $2$ spin${}^c$-cobordism ring has nontrivial $\beta$-torsion, it cannot be isomorphic to a polynomial $\mathbb{F}_2$-algebra. 

It follows as a trivial consequence of Theorem \ref{mspinc_subring} that the mod $(2,\beta)$ spin${}^c$-cobordism ring {\em still} cannot be a polynomial $\mathbb{F}_2$-algebra. One can, with a bit of calculation, deduce the same fact from the additive structure of $2$-local $MSpin^c$, by verifying that the Poincar\'{e} series of the mod $(2,\beta)$ spin${}^c$-cobordism ring is not the Poincar\'{e} series of any polynomial algebra. This avoids the use of our multiplicative methods. The advantage of our multiplicative methods is that we are able to prove that $\Omega^{Spin^c}_*/(2,\beta)$ is instead {\em uniformly $F$-isomorphic} to a polynomial algebra.

As far as we know, the terms ``$F$-isomorphism'' (perhaps better known as ``inseparable isogeny'') and ``uniform $F$-isomorphism'' originated with Quillen \cite{MR0298694}:
\begin{definition}\label{def of f-iso}
Given a prime $p$, a homomorphism of $\mathbb{F}_p$-algebras $f: A\rightarrow B$ is said to be an {\em $F$-isomorphism} if
\begin{itemize}
\item for each $a\in \ker f$, some power $a^n$ is zero, and
\item for each element $b\in B$, some power $b^{p^n}$ of $b$ is in the image of $f$.
\end{itemize}
The $F$-isomorphism $f$ is said to be {\em uniform} if $n$ can be chosen independently of $a$ and $b$.
\end{definition}
The notion of $F$-isomorphism is applied only to algebras over a field of positive characteristic, so we had better reduce modulo $2$ in order to apply this idea to the spin${}^c$-cobordism ring. We get a positive result:
\begin{mainthm}[Theorem \ref{f-iso thm}]\label{fiso_thm} 
The mod $2$ spin${}^c$-cobordism ring $\Omega^{Spin^c}_*\otimes_{\mathbb{Z}}\mathbb{F}_2$ is uniformly $F$-isomorphic to the graded $\mathbb{F}_{2}$-algebra
\begin{equation}
\label{ring 1} \mathbb{F}_{2}\left[\beta,y_{4i}, Z_{4j-2} : i\geq 1,\ j\geq 1,\ j\mbox{\ not\ a\ power\ of\ }2\right]/(\beta Z_{4j-2}),
\end{equation}
with $\beta$ the Bott element in degree $2$, with $y_{4i}$ in degree $4i$, and with $Z_{4j-2}$ in degree $4j-2$.
\end{mainthm}

\begin{maincor}\label{F-iso to a poly alg maincor}
The mod $(2,\beta)$ spin${}^c$-cobordism ring is uniformly $F$-isomorphic to a graded polynomial $\mathbb{F}_2$-algebra on 
\begin{itemize}
\item a generator in degree $4i$ for all positive integers $i$,
\item and a generator in degree $4j-2$ for all positive integers $j$ such that $j$ is not a power of $2$.
\end{itemize}
\end{maincor}

An $F$-isomorphism induces a homeomorphism on prime ideal spectra, so Theorem \ref{fiso_thm} yields a description of all prime ideals in the mod $2$ spin${}^c$ cobordism ring. That is, we have
\begin{maincor}[Corollary \ref{f-iso cor}]
The topological space $\Spec \Omega^{Spin^c}_*/(2)$ is homeomorphic to $\Spec$ of the $\mathbb{F}_2$-algebra \eqref{ring 1}. The topological space $\Spec \Omega^{Spin^c}_*/(2,\beta)$ is homeomorphic to $\Spec$ of the $\mathbb{F}_2$-algebra described in Corollary \ref{F-iso to a poly alg maincor}.
\end{maincor}

The last two sentences in Stong's 1968 book \cite{MR0248858} before the appendices begin are:
\begin{quotation}
One may relate the pair $(Spin, Spin^c)$ through exact sequences in precisely the same way as $(SU,U)$ are related (or as $(SO,O)$ are related). Computationally this is not of much use since one has no way to nicely describe the torsion in $\Omega^{Spin^c}_*$.
\end{quotation}
We regard Theorems \ref{large nonunital subring thm} and \ref{f-iso thm} as progress toward nicely describing the torsion in $\Omega^{Spin^c}_*$ by means of ring structure.

\subsection{Conventions}
\begin{itemize}
\item
Given a ring $R$ and symbols $x_1, \dots ,x_n$, we write $R\{ x_1, \dots ,x_n\}$ for the free $R$-module with basis $x_1, \dots ,x_n$.
\item We write $\beta$ for the Bott element in $\pi_2(ku)$, and also for its corresponding element $\beta = [\mathbb{C}P^1] \in \Omega^{Spin^c}_2$ under the Anderson--Brown--Peterson splitting of $2$-local $MSpin^c$.
\end{itemize}

\subsection{Funding}
The first author was partially supported by the electronic Computational Homotopy Theory (eCHT) research community, funded by National Science Foundation Research Training Group in the Mathematical Sciences grant 2135884.

\subsection{Acknowledgements}
The first author would like to thank Bob Bruner for many helpful conversations related to this work, and the Simons Foundation for providing the license for a copy of Magma \cite{MR1484478} used in calculations. We are also grateful to an anonymous referee for helpful comments.

\section{Preliminaries}\label{prelims}
In this section we present an extended review of some well-known facts about spin and spin${}^c$ cobordism, including the relationships various cobordism spectra, their homotopy groups, homology and cohomology groups, including the Steenrod algebra action on cohomology and the Pontryagin product in homology. This background material is necessary in order to understand the proofs of the results in the rest of the paper. Readers confident in their knowledge of this background material can skip to \cref{results}, where we begin proving new results.

\subsection{Review of the cohomology of the spectra $MSpin^{c}$ and $MSpin$}
There is an exact sequence of Lie groups
\begin{align*}
1 \longrightarrow U(1) \longrightarrow Spin^{c}(n) \longrightarrow SO(n) \longrightarrow 1
\end{align*}
that gives rise to the fiber sequence
\begin{align}\label{fib seq 1}
BU(1) \longrightarrow BSpin^{c} \longrightarrow BSO.
\end{align}
Using this fibration, Harada and Kono \cite{harada_kono} computed the mod 2 cohomology of the space $BSpin^c$:
\begin{theorem}\cite{harada_kono}
\begin{align}\label{harada kono iso}
H^{*}(BSpin^c;\mathbb{F}_{2}) &\cong \mathbb{F}_{2}(w_2,w_3,w_4,w_5,...)/I
\end{align}
where $I$ is the ideal $\langle w_3,\Sq^{2}(w_3),\Sq^2(Sq^{4}(w_3)),\Sq^8(Sq^4(Sq^2(w_3)),...\rangle$. 
\end{theorem}
The triviality of the ideal $I$ in the cohomology of $BSpin^c$ is a consequence of the first $d_2$ differential in the  Serre spectral sequence associated to the fiber sequence \eqref{fib seq 1}. It is not practical to write down a presentation for the $\mathbb{F}_2$-algebra $H^{*}(BSpin^c;\mathbb{F}_{2})$ which is more explicit than \eqref{harada kono iso}, since the difficulty of calculating iterated Steenrod squares applied to the Stiefel--Whitney class $w_3$ grows rapidly as the number of Steenrod squares grows. For example, $\Sq^8(\Sq^4(\Sq^2(w_3))$ has 38 monomials when expressed as a polynomial in the Stiefel--Whitney classes.

There is also an exact sequence 
\begin{align*}
     1 \longrightarrow \mathbb{Z}/2\mathbb{Z} \longrightarrow Spin(n) \longrightarrow SO(n) \longrightarrow 1 
\end{align*}
which gives rise to the fiber sequence 
\begin{align*}
    B\mathbb{Z}/2\mathbb{Z} \longrightarrow BSpin(n) \longrightarrow BSO(n).
\end{align*}
Using this, Quillen calculated: 
\begin{theorem}\cite{MR0290401}
   \begin{align*} H^{*}(BSpin;\mathbb{F}_{2}) \cong \mathbb{F}_{2}(w_2,w_3,w_4,w_5,...)/J
   \end{align*}
   where $J$ is the ideal $\langle w_2, w_3,\Sq^{2}(w_3),\Sq^2(Sq^{4}(w_3)),\Sq^8(Sq^4(Sq^2(w_3)),...\rangle$.
\end{theorem}

By the Thom isomorphism, we have that $H^*(MSpin^c;\mathbb{F}_2) \cong H^*(BSpin^c;\mathbb{F}_2)\{ U\}$ and $H^*(MSpin;\mathbb{F}_2) \cong H^*(BSpin;\mathbb{F}_2)\{ U\}$ as graded $\mathbb{F}_2$-vector spaces. Since \\$H^*(BSpin^c;\mathbb{F}_2)$ and $H^*(BSpin;\mathbb{F}_2)$ are quotients of $H^*(BO;\mathbb{F}_2) \cong \mathbb{F}_2[w_1, w_2, w_3, \dots]$, the action of Steenrod squares on $H^*(BSpin^c;\mathbb{F}_2)$ and $H^*(BSpin;\mathbb{F}_2)$ is determined by the Wu formula $\Sq^i w_j = \sum_{k=0}^i \binom{j+k-i-1}{k} w_{i-k}w_{j+k}$ and the Cartan formula. This, together with the formula $\Sq^n U = w^n U$ for the action of Steenrod squares on the Thom class $U$, determines the action of the Steenrod squares on the cohomology $H^*(MSpin^c;\mathbb{F}_2)$ of the spin${}^c$-bordism spectrum $MSpin^c$.
   
\subsection{Review of $MO_*$ and symmetric polynomials in the Stiefel--Whitney classes.}
\label{mo and combinatorics review}
The following definitions are classical (see e.g. chapter 1 of \cite{MR3443860}):
\begin{definition} Let $n$ be a nonnegative integer.
\begin{itemize}
\item Suppose $\lambda = (a_1, \dots ,a_n)$ is an unordered $n$-tuple of nonnegative integers. The {\em monomial symmetric polynomial associated to $\lambda$} is the symmetric polynomial \[ m_{\lambda}(X_1, \dots ,X_n) \in \mathbb{Z}[X_1, \dots ,X_n]^{\Sigma_n}\] which has the fewest nonzero monomial terms among all those which have $X_{1}^{a_1} X_{2}^{a_2}\dots X_{n}^{a_n}$ as a monomial term.
\item Given a nonnegative integer $m\leq n$, the {\em $m$th elementary symmetric polynomial} is the symmetric polynomial \[ e_{m}(X_1, \dots ,X_n) \in \mathbb{Z}[X_1, \dots ,X_n]^{\Sigma_n}\] given by
\begin{align*}
 e_m(X_1, \dots ,X_n) &= \sum_{1\leq d_1 < d_2 < \dots < d_m\leq n} X_{d_1}X_{d_2} \dots X_{d_m}.
\end{align*}
\end{itemize}
\end{definition}

The monomial symmetric polynomials form a $\mathbb{Z}$-linear basis for the ring of symmetric polynomials. The set of all finite products of elementary symmetric polynomials also famously (by Newton) forms a $\mathbb{Z}$-linear basis for the ring of symmetric polynomials. Consequently, for each $\lambda$, there exists a unique polynomial $P_{\lambda}(X_1, \dots ,X_n)$ such that 
\begin{align*}
 P_{\lambda}\left(e_1(X_1, \dots ,X_n), e_2(X_1, \dots ,X_n), \dots ,e_n(X_1, \dots ,X_n)\right) &= m_{\lambda}(X_1, \dots ,X_n).
\end{align*}
For more details about the polynomials $P_{\lambda}(X_1, \dots ,X_n)$, see the material on the transition matrix $M(m,e)$ and its inverse $M(e,m)$ in section 1.6 of \cite{MR3443860}, particularly (6.7)(i).

See \cite{MR0248858}, particularly pages 71 and 96 and surrounding material, for a nice exposition of the following result, which dates back to Thom \cite{MR0061823}: let $\Lambda$ be the set of unordered finite-length tuples of {\em positive} integers, each of which is not equal to $2^a-1$ for any integer $a$. Such integers are called ``non-dyadic,'' and such partitions are called ``non-dyadic partitions.'' For each $\lambda\in\Lambda$, write $\left|\lambda\right|$ for the length of $\lambda$, and write $\left| \left| \lambda \right|\right|$ for the sum of the elements of $\lambda$. Consider the polynomial 
\begin{align*}
 P_{\lambda}(w_1, \dots ,w_{\left| \lambda\right|})&\in H^*(BO;\mathbb{F}_2) \\
  &\cong \mathbb{F}_2[w_1, w_2, \dots ]\end{align*}
in the Stiefel--Whitney classes $w_1, w_2, \dots$. Then (see page 96 of \cite{MR0248858}, or pages 301-302 of \cite{MR0120654}) the set \[\{ P_{\lambda}(w_1, \dots ,w_{\left| \lambda\right|})U): \lambda\in\Lambda\}\] is a homogeneous $A$-linear basis for the graded free $A$-module $H^*(MO;\mathbb{F}_2)$, where $U\in H^0(MO;\mathbb{F}_2)$ denotes the Thom class, and $A$ is the mod $2$ Steenrod algebra.
Consequently, for each nonnegative integer $n$, $\pi_n(MO)$ is the $\mathbb{F}_2$-linear dual of the $\mathbb{F}_2$-vector space with basis the set 
\begin{equation}\label{thom basis} \left\{ P_{\lambda}(w_1, \dots ,w_{\left| \lambda\right|})U: \lambda \in \Lambda, \left|\left| \lambda\right|\right| = n\right\}.\end{equation}
Given two tuples $(a_1, \dots ,a_m)$ and $(b_1, \dots ,b_n)$, we have their concatenation
\begin{align*}
 (a_1, \dots ,a_m)\coprod (b_1, \dots ,b_n) 
  &:= (a_1 , \dots ,a_m,b_1, \dots ,b_n).
\end{align*}
The coproduct on $H^*(MO;\mathbb{F}_2)$ is then given by
\begin{align*}
\Delta(P_{\lambda}(w_1, \dots ,w_{\left| \lambda\right|})U)
 &= \sum_{\lambda^{\prime},\lambda^{\prime\prime}\in \Lambda :\ \lambda = \lambda^{\prime}\coprod\lambda^{\prime\prime}} P_{\lambda^{\prime}}(w_1, \dots ,w_{\left| \lambda^{\prime}\right|})U \otimes  P_{\lambda^{\prime\prime}}(w_1, \dots ,w_{\left| \lambda^{\prime\prime}\right|})U.\end{align*}
Consequently, if we write $\{ Y_{\lambda}\}$ for the basis of $\pi_*(MO)$ dual to the basis \eqref{thom basis}, then $Y_{\lambda}Y_{\lambda^{\prime}} = Y_{\lambda\coprod\lambda^{\prime}}$, and $\Omega^O_*\cong \pi_*(MO)\cong \mathbb{F}_2[ Y_{(2)},Y_{(4)},Y_{(5)},Y_{(6)},Y_{(8)}, \dots]$, with $Y_{(i)}$ in degree $i$. We will sometimes write $Y_i$ as an abbreviation for Thom's generator $Y_{(i)}$ of $\Omega^O_*$.

\subsection{Maps between bordism theories}

The first stages of the  Whitehead tower for the orthogonal group are: 
\begin{align*} 
BSpin = BO\langle 4 \rangle \longrightarrow BSO = BO\langle 2 \rangle \longrightarrow BO.
\end{align*}
While $BSpin^{c}$ does not fit into this sequence via a connective cover, the map $BSpin \longrightarrow BSO$ factors through $BSpin^c$. 
There is a commutative diagram whose rows and columns are fiber sequences: 
\begin{equation*}
\begin{tikzcd}[row sep=huge]
K(\mathbb{Z}/2\mathbb{Z},0) \arrow[r] \arrow[d] &
Spin(n) \arrow[r] \arrow[d,swap] &
SO(n) \arrow[d,]
\\
K(\mathbb{Z},1) \arrow[r] \arrow[d] &  Spin^{c}(n) \arrow[r] \arrow[d] & SO(n) \arrow[d] \\
K(\mathbb{Z},1)\cong U(1) \arrow[r] & U(1) \arrow[r] & *
\end{tikzcd}
\end{equation*}
On the level of spectra, we have maps
\begin{align}
MSpin \longrightarrow MSpin^{c} \longrightarrow MSO \longrightarrow MO.
\end{align}
The maps induced in homotopy give the maps of respective cobordism rings.  We will specifically consider the images of $MSpin^{c}_{*}$ and $MSpin_*$ in $MO_{*}$.

By the Anderson--Brown--Peterson splitting of $2$-local $MSpin^c$, the cohomology $H^*(MSpin^c;\mathbb{F}_2)$ splits as a direct sum of suspensions of $H^*(ku;\mathbb{F}_2) \cong A//E(1)$ and of $H^*(H\mathbb{F}_2;\mathbb{F}_2) \cong A$. Here we are using the standard notation $A$ for the mod $2$ Steenrod algebra, and $A//E(1)$ for its quotient $A\otimes_{E(1)}\mathbb{F}_2$, where $E(1)$ is the subalgebra of $A$ generated by $\Sq^1$ and by $Q_1 = [\Sq^1,\Sq^2]$. Hence the $s=0$-line in the $2$-primary Adams spectral sequence for $MSpin^c$,
\begin{align}
\label{adams ss 1} E_2^{s,t} \cong \Ext_A^{s,t}\left(H^*(MSpin^c;\mathbb{F}_2),\mathbb{F}_2\right) & \Rightarrow \pi_{t-s}(MSpin^c)^{\wedge}_2 \\
\nonumber d_r : E_r^{s,t} & \rightarrow E_r^{s+r,t+r-1} 
\end{align}
is a direct sum of suspensions of $\mathbb{F}_2$, with one summand $\Sigma^t\mathbb{F}_2$ for each summand $\Sigma^t ku_{(2)}$ in $MSpin^c_{(2)}$, and also with one summand $\Sigma^t\mathbb{F}_2$ for each summand $\Sigma^t H\mathbb{F}_2$ in $MSpin^c_{(2)}$. 

Anderson--Brown--Peterson prove in \cite{MR0219077} that all differentials in the Adams spectral sequence \eqref{adams ss 1} are zero. Consequently the $s=0$-line $\hom_A(H^*(MSpin^c;\mathbb{F}_2),\mathbb{F}_2)$ is the reduction of $\pi_*(MSpin^c)$ modulo the ideal generated by $2$ and by the Bott element $\beta\in \pi_2(ku)$. 

This means we can calculate $\pi_*(MSpin^c)/(2,\beta)$ simply by calculating\linebreak $\hom_A(H^*(MSpin^c;\mathbb{F}_2),\mathbb{F}_2)$, i.e., the $A_*$-comodule primitives $\mathbb{F}_2\Box_{A_*} H_*(MSpin^c;\mathbb{F}_2)$. The advantage of thinking in terms of comodule primitives is that the Adams spectral sequence respects ring structure: if we calculate the homology $H_*(MSpin^c;\mathbb{F}_2)$ as a ring, then by simply restricting to the comodule primitives in $H_*(MSpin^c;\mathbb{F}_2)$, we have calculated $\Omega^{Spin^c}_*/(2,\beta)$.  

The same remarks apply {\em mutatis mutandis} for the spin bordism spectrum $MSpin$, the oriented bordism spectrum $MSO$ or for the unoriented bordism spectrum $MO$ in place of $MSpin^c$. The Anderson--Brown--Peterson splitting for $MSpin$ is as a wedge of suspensions of $ko$, $ko\langle 2 \rangle$, and $H\mathbb{F}_{2}$. The analogue of the Anderson--Brown--Peterson splitting for $MO$ is Thom's splitting of $MO$ as a wedge of suspensions of $H\mathbb{F}_2$, while $MSO_{(2)}$ splits as a wedge of suspensions of $H\mathbb{Z}_{(2)}$ and $H\mathbb{F}_2$; as far as we know, the latter splitting was originally proven by Wall \cite{MR0120654}. Since $H^*(MSpin^c;\mathbb{F}_2)$ is a quotient $A$-module of $H^*(MSO;\mathbb{F}_2)$, which is in turn a quotient $A$-module of $H^*(MO;\mathbb{F}_2)$, dualizing yields that $H_*(MSpin^c;\mathbb{F}_2)$ is a subcomodule of $H_*(MSO;\mathbb{F}_2)$, which is in turn a subcomodule of $H_*(MO;\mathbb{F}_2)$.  

Hence our broad strategy for calculating $\Omega^{Spin^c}_*/(2,\beta) \cong MSpin^c_*/(2,\beta)$ and $\Omega^{Spin}_{*}/(2,\eta,\alpha,\beta) \cong MSpin_*/(2,\eta,\alpha,\beta)$, and the natural maps $MSpin_*/(2,\eta,\alpha,\beta)\rightarrow MSpin^c_*/(2,\beta)\rightarrow MO_*$, is to calculate the $A_*$-comodule primitives in $H_*(MSpin^c;\mathbb{F}_2)$ and in $H_*(MSpin;\mathbb{F}_2)$, regarding each as $A_*$-subcomodule algebras of $H_*(MO;\mathbb{F}_2)$. The resulting information will describe $\Omega^{Spin^c}_*/(2,\beta)$ as a subring of $\Omega^O_*$. 
Details of this strategy are given in the description of the computational method in the proof of Proposition \ref{mspinc_image}.

The relationships between the spin,  spin${}^c$, oriented, and unoriented cobordism rings and their homologies is summarized in the following diagram, in which hooked arrows represent one-to-one maps:
\[\begin{tikzcd} 
    {\text{MSpin}}_{*}/(2,\eta,\alpha,B) \pgfmatrixnextcell {H_{*}(\text{MSpin};\mathbb{F}_{2})} \\
    \\
	{\text{MSpin}^{c}_{*}/(2,\beta) } \pgfmatrixnextcell {H_{*}(\text{MSpin}^{c};\mathbb{F}_{2})} \\
	\\
	{\text{MSO}_{*}/(2) }  \pgfmatrixnextcell{H_{*}(\text{MSO};\mathbb{F}_{2})} \\
	\\
	{\text{MO}_{*} }  \pgfmatrixnextcell {H_{*}(\text{MO};\mathbb{F}_{2})}
	\arrow[shorten <=3pt, shorten >=3pt, hook', from=1-1, to=3-1]
	\arrow[shorten <=3pt, shorten >=3pt, hook', from=3-1, to=5-1]
	\arrow[shorten <=4pt, shorten >=4pt, from=1-1, to=1-2]
	\arrow[shorten <=5pt, shorten >=5pt, from=3-1, to=3-2]
	\arrow[shorten <=13pt, shorten >=7pt, from=5-1, to=5-2]
	\arrow[shorten <=3pt, shorten >=3pt, hook', from=1-2, to=3-2]
	\arrow[shorten <=3pt, shorten >=3pt, hook', from=3-2, to=5-2]	\arrow[shorten <=3pt, shorten >=3pt, hook', from=5-2, to=7-2]
    \arrow[shorten <=3pt, shorten >=3pt, hook', from=5-1, to=7-1]
     \arrow[shorten <=15pt, shorten >=10pt, hook', from=7-1, to=7-2]
\end{tikzcd}\]


\section{The $Spin^{c}$ bordism ring in low degrees} 
\label{results}
In the statement of Proposition \ref{mspinc_image}, 
we use Thom's presentation $\mathbb{F}_{2}[Y_{2},Y_{4},Y_{5},...]$ for the unoriented cobordism ring $\Omega^O_*$.
\begin{proposition}\label{mspinc_image} The image of the map $\Omega^{Spin^c}_{*} \longrightarrow \Omega^O_*$ agrees, in degrees $\leq 33$, with the subring of $\Omega^O_*$ generated by the elements
\[
    Y_{2}^{2},Y_{4}^{2},Y_{5}^{2},Y_{6}^{2}, Y_{9}^{2}, Y_{10}^{2}, Y_{11}^{2}, Y_{12}^{2}, Y_{13}^{2}, Y_{14}^{2}, Y_{15}^{2}, Y_{16}^{2},T_{24},T_{29},T_{31},T_{32},\mbox{\ and\ }T_{33} ,
\]
where 
\begin{align}
\label{t24 eq} T_{24} &= Y_{14} Y_5^2 +
 Y_{13} Y_{11} +
 Y_{13} Y_9 Y_2 +
 Y_{13} Y_6 Y_5 +
 Y_{13} Y_5 Y_2^3 +
 Y_{12} Y_5^2 Y_2 +
 Y_{11}^2 Y_2 \\ \nonumber & \ \ \ +
 Y_{11} Y_9 Y_4 +
 Y_{11} Y_8 Y_5 +
 Y_{11} Y_6 Y_5 Y_2 +
 Y_{11} Y_5 Y_4^2 +
 Y_{11} Y_5 Y_4 Y_2^2 +
 Y_{10} Y_5^2 Y_4 \\ \nonumber & \ \ \ +
 Y_{10} Y_5^2 Y_2^2 +
 Y_9^2 Y_4 Y_2 +
 Y_9^2 Y_2^3 +
 Y_9 Y_8 Y_5 Y_2 +
 Y_9 Y_6 Y_5 Y_4 +
 Y_9 Y_6 Y_5 Y_2^2 \\ \nonumber & \ \ \ +
 Y_9 Y_5^2 +
 Y_9 Y_5 Y_4^2 Y_2 +
 Y_6^2 Y_5^2 Y_2 +
 Y_5^4 Y_4 +
 Y_5^2 Y_4^3 Y_2 +
 Y_5^2 Y_4^2 Y_2^3, \end{align}
\begin{align}
\nonumber T_{29}&= Y_{19} Y_5^2 +
 Y_{17} Y_5^2 Y_2 +
 Y_{14} Y_5^3 +
 Y_{13} Y_6 Y_5^2 +
 Y_{13} Y_5^2 Y_2^3 +
 Y_{11} Y_9^2 +
 Y_{11} Y_8 Y_5^2  \\ \nonumber & \ \ \ +
 Y_{11} Y_5^2 Y_4 Y_2^2 +
 Y_{10} Y_9 Y_5^2 +
 Y_{10} Y_5^3 Y_2^2 +
 Y_9^3 Y_2 +
 Y_9^2Y_6 Y_5 +
 Y_9^2 Y_5 Y_2^3  \\ \nonumber & \ \ \ +
 Y_9 Y_6 Y_5^2 Y_2^2 +
 Y_9 Y_5^4 +
 Y_6^2 Y_5^3 Y_2 +
 Y_5^5 Y_4 +
 Y_5^3 Y_4^2 Y_2^3, \end{align}
\begin{align}
\nonumber T_{31} &=  Y_{21} Y_5^2 +
 Y_{19} Y_5^2 Y_2 +
 Y_{17} Y_5^2 Y_4 +
 Y_{17} Y_5^2 Y_2^2 +
 Y_{16} Y_5^3 +
 Y_{13} Y_9^2 +
 Y_{13} Y_8 Y_5^2   \\ \nonumber & \ \ \ +
 Y_{13} Y_5^2 Y_4 Y_2^2 +
 Y_{12} Y_9 Y_5^2 +
 Y_{12} Y_5^3 Y_2^2 +
 Y_{11} Y_{10} Y_5^2 +
 Y_{11} Y_9^2 Y_2 +
 Y_{11} Y_6 Y_5^2 Y_2^2  \\ \nonumber & \ \ \ +
 Y_9^3 Y_4 +
 Y_9^3 Y_2^2 +
 Y_9^2 Y_8 Y_5 +
 Y_9^2 Y_5 Y_4^2 +
 Y_9^2 Y_5 Y_4 Y_2^2 +
 Y_9^2 Y_5 Y_2^4   \\ \nonumber & \ \ \ +
 Y_9 Y_8 Y_5^2 Y_2^2 +
 Y_9 Y_6^2 Y_5^2 +
 Y_9 Y_5^2 Y_4^2 Y_2^2+
 Y_8^2 Y_5^3 +
 Y_6^2 Y_5^3 Y_4 +
 Y_5^3 Y_4^3 Y_2^2 +
 Y_5^3 Y_4^2 Y_2^4,\end{align}\begin{align}
\nonumber T_{32}&= Y_{22} Y_5^2+
 Y_{21} Y_{11}+
 Y_{21} Y_9 Y_2+
 Y_{21} Y_6 Y_5+
 Y_{21} Y_5 Y_2^3+
 Y_{20} Y_5^2 Y_2+
Y_{19} Y_{13} \\ \nonumber & \ \ \ +
Y_{19} Y_9 Y_4+
Y_{19} Y_8 Y_5+
Y_{19} Y_6 Y_5 Y_2+
Y_{19} Y_5 Y_4^2+
Y_{19} Y_5 Y_4 Y_2^2+
Y_{18} Y_5^2 Y_4 \\ \nonumber & \ \ \ +
Y_{18} Y_5^2 Y_2^2+
Y_{17} Y_{13} Y_2+
Y_{17} Y_{11} Y_4+
Y_{17} Y_8 Y_5 Y_2+
Y_{17} Y_6 Y_5 Y_4+
Y_{17} Y_6 Y_5 Y_2^2 \\ \nonumber & \ \ \ +
Y_{17} Y_5^3+
Y_{17} Y_5 Y_4^2 Y_2+
Y_{16} Y_{11} Y_5+
Y_{16} Y_9 Y_5 Y_2+
Y_{14} Y_{13} Y_5+
Y_{14} Y_{11} Y_5 Y_2 \\ \nonumber & \ \ \ +
Y_{14} Y_9^2+
Y_{14} Y_9 Y_5 Y_4+
Y_{14} Y_9 Y_5 Y_2^2+
Y_{13}^2 Y_6+
Y_{13}^2 Y_2^3+
Y_{13} Y_{11} Y_6 Y_2 \\ \nonumber & \ \ \ +
Y_{13} Y_{10} Y_9+
Y_{13} Y_{10} Y_5 Y_2^2+
Y_{13} Y_9 Y_8 Y_2+
Y_{13} Y_9 Y_6 Y_4+
Y_{13} Y_6^2 Y_5 Y_2+
Y_{13} Y_5^3Y_4 \\ \nonumber & \ \ \ +
Y_{12} Y_{11}  Y_9 +
Y_{12} Y_{11} Y_5 Y_2^2+
Y_{12} Y_{10} Y_5^2+
Y_{12} Y_9 Y_6 Y_5+
Y_{12} Y_5^4+
Y_{11}^2 Y_{10} \\ \nonumber & \ \ \ +
Y_{11}^2 Y_8 Y_2+
Y_{11}^2 Y_6 Y_2^2+
Y_{11}^2 Y_4 Y_2^3+
Y_{11} Y_{10} Y_6 Y_5+
Y_{11} Y_9 Y_8 Y_4+
Y_{11} Y_9 Y_6^2 \\ \nonumber & \ \ \ +
Y_{11} Y_6^2 Y_5 Y_4+
Y_{11} Y_6 Y_5^3+
Y_{10}^2 Y_5^2 Y_2+
Y_{10} Y_9 Y_8 Y_5+
Y_{10} Y_9 Y_5 Y_4^2+
Y_{10} Y_6^2 Y_5^2 \\ \nonumber & \ \ \ +
Y_9^3 Y_5+
Y_9^2 Y_8 Y_6+
Y_9^2 Y_6 Y_4^2+
Y_9^2 Y_5^2 Y_4+
Y_9 Y_8^2 Y_5 Y_2+
Y_9 Y_8 Y_5^3+
Y_9 Y_6^3 Y_5 \\ \nonumber & \ \ \ +
Y_8^2 Y_5^2 Y_4 Y_2+
Y_8^2 Y_5^2 Y_2^3,\ \ \ \mbox{and}\end{align}\begin{align}
\nonumber T_{33} &=Y_{23} Y_5^2+
 Y_{21} Y_5^2 Y_2+
Y_{19} Y_5^2 Y_4+
Y_{18} Y_5^3+
Y_{17} Y_6 Y_5^2+
Y_{14} Y_9 Y_5^2+
Y_{13} Y_{10} Y_5^2 \\ \nonumber & \ \ \ +
Y_{13} Y_5^2 Y_4^2 Y_2+
Y_{12} Y_{11} Y_5^2+
Y_{11} Y_{11} Y_{11}+
Y_{11} Y_{11} Y_9 Y_2+
Y_{11} Y_{11} Y_6 Y_5+
Y_{11} Y_{11} Y_5 Y_2^3 \\ \nonumber & \ \ \ +
Y_{11} Y_5^2 Y_4^2 Y_4+
Y_{10} Y_5^3 Y_4^2+
Y_9^2 Y_5^3+
Y_9 Y_6 Y_5^2 Y_4^2+
Y_8^2 Y_5^3 Y_2+
 Y_5^5 Y_4^2+
 Y_5^3 Y_4^4 Y_2.
\end{align}
\end{proposition}
Each of the elements $T_i$ defined in Proposition \ref{mspinc_image} is a linear combination of monomials in $\Omega^O_*$. Those monomials are generally not {\em individually} members of $\Omega^{Spin^c}_*$: for example, $Y_{14} Y_5^2\in \Omega^O_{24}$ does not lift to an element of $\Omega^{Spin^c}_{24}$, even though a linear combination of $Y_{14} Y_5^2$ with other monomials in degree $24$ {\em does} lift to the element $T_{24}\in \Omega^{Spin^c}_{24}$. 
\begin{proof}[Proof of Proposition \ref{mspin_image}]
This proposition is proven using computer calculation. We will describe our method for calculating $\im (\Omega^{Spin^c}_* \rightarrow \Omega^O_*)$ in degrees $\leq d$, for any fixed choice of $d$. 
The first author wrote a Magma \cite{MR1484478} program which implements this method, and we have made its source code available at \url{https://github.com/hassan-abdallah/spinc_cobordism}. Once the reader is convinced of the correctness of the method, the proof of this proposition consists of simply running the calculation through degree $33$, either by using our software, or by writing their own software implementation of the method, if desired.

We freely use the relationship between the spin${}^c$-cobordism ring, the unoriented cobordism ring, and the mod $2$ cohomology of $BO$ detailed in \cref{prelims}. Let $\lambda$ be a non-dyadic partition of a nonnegative integer $n$.
We want to know whether its corresponding element $Y_{\lambda} \in MO_{n}$ is in the image of the map $MSpin^{c}_{n}\rightarrow MO_n$. The element $Y_{\lambda} \in MO_{n}$ has a Hurewicz image, i.e., the image of $Y_{\lambda}$ under the Hurewicz map $\pi_n(MO)\rightarrow H_n(MO;\mathbb{F}_2)$. In \cref{mo and combinatorics review}, we described the dual element $P_{\lambda}U\in H^n(MO;\mathbb{F}_2)$ to the Hurewicz image of $Y_{\lambda}$, using Thom's basis for $MO_*$. 
The element $P_{\lambda}U$ can be written as $U$ times a polynomial in the Stiefel--Whitney classes by applying an appropriate transition matrix\footnote{We remark that the computation of this transition matrix is one of the most computationally expensive parts of this process, despite its being a simple combinatorial problem. It is in fact the inverse of a Kostka matrix \cite{MR2868112}.}. Once $P_{\lambda}U$ is calculated, we see that $Y_{\lambda}$ is in the image if and only if, when reduced modulo $w_1$ and the relations in the $H^*(BO; \mathbb{F}_2)$-module $H^{*}(MSpin^{c};\mathbb{F}_{2})$,
$P_{\lambda}U$ is an $A$-module primitive in $H^{*}(MSpin^{c};\mathbb{F}_{2})$.

Consequently our method for calculating $\im (\Omega^{Spin^c}_* \rightarrow \Omega^O_*)$ is merely a method for building up a basis for the $\mathbb{F}_2$-vector space of $A$-module primitives through some fixed degree $d$, {\em in terms of non-dyadic partitions.}
We work one degree at a time, but via induction, assuming we have already completed the calculation at all lower degrees.

The induction begins at degree $0$, where there is nothing to say: the empty partition $\varnothing$ yields the unique $A$-module primitive $\varnothing\cdot U$ in $H^0(MSpin^c; \mathbb{F}_2)$. For each integer $n\in [1,d]$, the product $\Sq^n\cdot \varnothing\cdot U \in H^n(MSpin^c; \mathbb{F}_2)$ is simply $w_nU$ modulo the relations in $H^*(MSpin^c; \mathbb{F}_2)$, by the classical formula $\Sq^nU = w_nU$ for the action of Steenrod squares on the Thom class in $H^0(MO;\mathbb{F}_2)$. Record the elements $\{ \Sq^1U, \Sq^2U, \dots ,\Sq^dU\}$ in an unordered list $D$. Here the symbol $D$ stands for ``decomposable,'' as we will use it to build up a list of $A$-module decomposables in $H^*(MSpin^c;\mathbb{F}_2)$ in degrees $\leq d$. 

We are not done with the initial step in the induction: for each nonzero element $\Sq^nU$, we calculate $\{ \Sq^1\Sq^nU, \Sq^2\Sq^nU, \dots , \Sq^{d-n}\Sq^nU\}$ using the Thom formula $\Sq^nU = w_nU$ and the Wu formula (\cite{MR0035993}, but see \cite[pg. 94]{MR0440554} for a textbook reference) for the action of $\Sq^n$ on Stiefel--Whitney classes, and we include the results in $D$. We keep going: calculate all the three-fold composites of Steenrod squares applied to $U$ {\em landing in degrees $\leq d$}, then all the four-fold composites of Steenrod squares applied to $U$ {\em landing in degrees $\leq d$}, and so on. We emphasize the phrase ``landing in degrees $\leq d$'' because it is what ensures that this calculation eventually terminates! 

After we are done with that inductive calculation, $D$ now contains an $\mathbb{F}_2$-linear basis for the $A$-submodule of $H^*(MSpin^c;\mathbb{F}_2)$ generated by the Thom class $U$, in all degrees $\leq d$.

Now we are ready for the inductive step:
\begin{description}
\item[Inductive hypothesis at the $n$th step] 
We have produced a list $G$ of $\mathbb{F}_2$-linear combinations of non-dyadic partitions of degree $< n$, such that the $\mathbb{F}_2$-linear span of $\{ P_{\lambda}U: \lambda\in G\}$ is a basis for the set of $A$-module primitives in $H^*(MSpin^c;\mathbb{F}_2)$ in all degrees $<n$. 
We have also produced a list $D$ of $\mathbb{F}_2$-linear combinations of non-dyadic partitions of degree $\leq d$, such that the $\mathbb{F}_2$-linear span of $\{ P_{\lambda}U: \lambda\in D\}$ is precisely the $A$-submodule of $H^*(MSpin^c;\mathbb{F}_2)$ generated by $G$ in degrees $\leq d$.
\item[Calculation for the $n$th inductive step]
Write $D_n$ for the $\mathbb{F}_2$-linear span of the degree $n$ elements in $D$.
Calculate an $\mathbb{F}_2$-linear basis $B$ for\linebreak $H^n(MSpin^c;\mathbb{F}_2)/D_n$. 
Let $G^{\prime}$ be $G \cup B$. Use the calculated transition matrix to convert the members of $G^{\prime}$ from the Thom/partition basis to the Stiefel--Whitney monomial basis, and then use the Thom formula and the Wu formula to calculate all Steenrod squares on the members of $G^{\prime}$, then all Steenrod squares on those, etc., in degrees $\leq d$. Use the transition matrix to convert back to the partition basis, and $D^{\prime}$ for the resulting list of linear combinations of non-dyadic partitions. Now we are ready to iterate, with $G^{\prime}$ in place of $G$, and with $D^{\prime}$ in place of $D$.
\end{description}

Once we complete the $n=d$ step, we have an $\mathbb{F}_2$-linear basis for the image of the map $\Omega^{Spin^c}_*\rightarrow \Omega^O_*$ in all degrees $\leq d$, expressed in terms of Thom's partition basis for $\Omega^O_*$. Consequently we have a description of $\Omega^{Spin^c}_*/(2,\beta)$, in all degrees $\leq d$, as a subring of $\Omega^O_*\cong \mathbb{F}_2[Y_2, Y_4, Y_5, Y_6, Y_8, Y_9, \dots ]$. 
\end{proof}
In principle there is no obstruction to using the same method to make calculations of products in $\Omega^{Spin^c}_*$ in degrees $>33$. We stopped at degree $33$ simply because, around the time we completed degree $33$, we could see enough of the ring structure of $\Omega^{Spin^c}_*$ to prove that the mod $(2,\beta)$ spin${}^c$-cobordism ring is not a polynomial algebra, and to suggest the right statements for Proposition \ref{hassans conj} and Theorem \ref{fiso_thm}. The products in $\Omega^{Spin^c}_{*}$ required for the proof of Theorem \ref{milnor mfld thm} are known as soon as one computes the ring structure through degree $24$. Completing the calculation through degree 33 took more than a week on consumer hardware. We anticipate that a calculation of even a few degrees beyond 33 will potentially take months to complete without the use of high-performance computing hardware or an improved algorithm. 

The same computational method described in the proof of Proposition \ref{mspinc_image}, applied to $MSpin$ rather than $MSpin^c$, yields:
 \begin{proposition}\label{mspin_image} 
 The image of the map $\Omega^{Spin}_{*} \longrightarrow \Omega^O_*$ agrees, in degrees $\leq 31$, with the subring of $MO_*$ generated by the elements
\begin{align*}
   \langle Y_2^4, Y_5^2, Y_4^4, 
    Y_9^2 + Y_5^2 Y_4^2, Y_{11}^2 + Y_9^2 Y_2^2 + Y_5^2 Y_4^2 Y_2^2, Y_{13}^2 + Y_{11}^2 Y_2^2 + Y_9^2 Y_4^2 + Y_8^2 Y_5^2,Y_6^4\\ 
    T_{24} + Y_{12}^{2}+Y_{10}^{2}Y_{2}^{2}+Y_{8}^{2}Y_{4}^{2}+Y_{8}^2Y_{2}^4+Y_{6}^{2}Y_{4}^{2}Y_{2}^{2} + Y_{5}^{4}Y_{2}^{2} +Y_4^{6}, T_{29}
\rangle.
 \end{align*}
\end{proposition}


While the individual element $Y_i \in \Omega^O_{i}$ does not generally lift to $\Omega^{Spin^c}_i$, its square $Y_i^2\in \Omega^O_{2i}$ does lift to the element $Z_{2i}\in \Omega^{Spin^c}_{2i}$. Consequently {\em the squares of each of the monomials in each of the elements $T_i$ lift to $\Omega^{Spin^c}_*$}. For $i=24,29,31,32$, and $33$, let $U_{2i}$ denote the element of $\Omega^{Spin^c}_*$ obtained by taking the definition of $T_i$ in Proposition \ref{mspinc_image} and replacing each instance of $Y_n$ with $Z_{2n}$. For example, \eqref{t24 eq} yields that
\begin{align*}
 U_{48} &= Z_{28} Z_{10}^2 +
 Z_{26} Z_{22} +
 Z_{26} Z_{18} Z_4 +
 Z_{26} Z_{12} Z_{10} +
 Z_{26} Z_{10} Z_4^3 +
 Z_{24} Z_{10}^2 Z_4 +
 Z_{22}^2 Z_4 \\ & \ \ \ +
 Z_{22} Z_{18} Z_8 +
 Z_{22} Z_{16} Z_{10} +
 Z_{22} Z_{12} Z_{10} Z_4 +
 Z_{22} Z_{10} Z_8^2 +
 Z_{22} Z_{10} Z_8 Z_4^2  \\ & \ \ \ +
 Z_{20} Z_{10}^2 Z_8 +
 Z_{20} Z_{10}^2 Z_4^2 +
 Z_{18}^2 Z_8 Z_4 +
 Z_{18}^2 Z_4^3 +
 Z_{18} Z_{16} Z_{10} Z_4 +
 Z_{18} Z_{12} Z_{10} Z_8  \\ & \ \ \ +
 Z_{18} Z_{12} Z_{10} Z_4^2 +
 Z_{18} Z_{10}^2 +
 Z_{18} Z_{10} Z_8^2 Z_4 +
 Z_{12}^2 Z_{10}^2 Z_4 +
 Z_{10}^4 Z_8 +
 Z_{10}^2 Z_8^3 Z_4  \\ & \ \ \ +
 Z_{10}^2 Z_8^2 Z_4^3.
\end{align*}
Then, as a consequence of Proposition \ref{mspinc_image}, we have:
\begin{theorem}\label{mspinc mod 2 and beta}
    The subring of $\Omega^{Spin^c}_{*}/(2,\beta)$ generated by all homogeneous elements of degree $\leq 33$ is isomorphic to:
\begin{align*}
\mathbb{F}_{2}[Z_4,Z_8,Z_{10},Z_{12},Z_{16},Z_{18},Z_{20},Z_{22},Z_{24},Z_{26}, Z_{28},Z_{32}, T_{24}, T_{29}, T_{31}, T_{32}, T_{33}]/I,
\end{align*}
where $I$ is the ideal generated by $T_{24}^{2}-U_{48},\ T_{29}^{2}-U_{58},\ T_{31}^{2}-U_{62},\ T_{32}^{2} - U_{64},$ and $T_{33}^{2} - U_{66}$.
\end{theorem}
The relations $T_i^2 = U_{2i}$, with $T_i$ indecomposable in $\Omega^{Spin^c}_i$ and with $U_{2i}$ a polynomial in the indecomposable elements $Z_n$, immediately implies that $\Omega^{Spin^c}_{*}/(2,\beta)$ is not a polynomial algebra.

\begin{theorem}\label{mspinc_subring in main text}
The subring of the mod $2$ spin${}^c$-cobordism ring $\Omega^{Spin^c}_*\otimes_{\mathbb{Z}}\mathbb{F}_2$ generated by all homogeneous elements of degree $\leq 33$ is isomorphic to:
\begin{align*}
\mathbb{F}_{2}[\beta,Z_4,Z_8,Z_{10},Z_{12},Z_{16},Z_{18},Z_{20},Z_{22},Z_{24},Z_{26}, Z_{28},Z_{32}, T_{24}, T_{29}, T_{31}, T_{32}, T_{33}]/I,
\end{align*}
where $I$ is the ideal generated by
\begin{itemize}
\item $\beta Z_{i}$ for each $i\equiv 2 \mod 4$,
\item and $\beta T_{i}$ and $T_i^2 - U_{2i}$ for $i=24,29,31,32,33$.
\end{itemize}
\end{theorem}
\begin{proof}
Let $T$ denote the ideal of $\Omega^{Spin^c}_*$ consisting of $2$-torsion elements, and let $\tilde{T}$ denote the kernel of the ring map 
\begin{align*}
\Omega^{Spin^c}_*\otimes_{\mathbb{Z}}\mathbb{F}_2
 &\rightarrow \left(\Omega^{Spin^c}_*/T\right)\otimes_{\mathbb{Z}}\mathbb{F}_2.\end{align*} 
The ring $\left(\Omega^{Spin^c}_*/T\right)\otimes_{\mathbb{Z}}\mathbb{F}_2$ was calculated by Stong \cite[Proposition 11]{MR192516}: it is a polynomial $\mathbb{F}_2[\beta]$-algebra on generators in degrees $4,8,12,16,\dots$. 
Since $\tilde{T}$ is an ideal in $\Omega^{Spin^c}_*$, to calculate the product in the ring $\Omega^{Spin^c}_*\otimes_{\mathbb{Z}}\mathbb{F}_2$, it suffices to calculate 
\begin{itemize}
\item the products between generators of $\tilde{T}$,
\item and the products between generators of $\tilde{T}$ and lifts, to $\Omega^{Spin^c}_*\otimes_{\mathbb{Z}}\mathbb{F}_2$ of generators of $\left(\Omega^{Spin^c}_*/T\right)\otimes_{\mathbb{Z}}\mathbb{F}_2$.
\end{itemize}
Both of these types of products land in $\tilde{T}$. Since $\tilde{T}$ maps injectively under the map $\Omega^{Spin^c}_*\rightarrow \Omega^O_*$, we can embed $\Omega^{Spin^c}_*/(2,\beta)$ into $\Omega^O_*$ and bring to bear our calculations of the image of this map, from Proposition \ref{mspinc_image}. All we need to do is to determine, in our set of generators for $\Omega^{Spin^c}_*/(2,\beta)$ through degree $33$, a maximal set of linear combinations of products of generators which generate $\beta$-torsion elements in $\Omega^{Spin^c}_*/(2)$, i.e., copies of $H\mathbb{F}_2$ rather than $ku_{(2)}$ in the Anderson--Brown--Peterson splitting. 

As a consequence of the Anderson--Brown--Peterson splitting, generators of\linebreak $\Omega^{Spin^c}_*/(2,\beta)$ that are 2-torsion, and hence $\beta$-torsion, in $\Omega^{Spin^c}_*$ are those whose corresponding $A$-module primitive in $H^*(MSpin^c;\mathbb{F}_{2})$ is $\textit{not}$ $Q_0$ or $Q_1$ torsion. We identify such generators by re-running the entire process from the proof of Proposition \ref{mspinc_image}, but with the following modification: at the start of the calculation, before the induction on degree, we begin by letting $D$ be a list of {\em all Stiefel--Whitney monomials in $H^*(MSpin^c; \mathbb{F}_2)$ in degrees $\leq d$ which are $(Q_0,Q_1)$-torsion,} together with all words in the Steenrod squares applied to such $(Q_0,Q_1)$-torsion Stiefel--Whitney monomials\footnote{In principle, this step in the calculation could go wrong: see Remark \ref{remark on simplification in pf} for the issue, and why we know that in fact it does not cause a problem.}, instead of letting $D$ begin as the empty set. Consequently, as we proceed through the induction, $D$ is not only the set of $A$-module decomposables, but also the set of Stiefel--Whitney monomials which generate copies of $A//E(1)$. 

Re-running our inductive calculation from Proposition \ref{mspinc_image}, but with this initial list for $D$, yields a set of $A$-module generators for $H^*(MSpin^c;\mathbb{F}_2)$ {\em modulo $(Q_0,Q_1)$-torsion}. Comparison of the lists produced by the first calculation and the second calculation, then using the translation matrix to translate back from the basis of Stiefel--Whitney monomials to the dual basis of partitions (i.e., Thom's basis for $\Omega^O_*$), 
 gives us a set of generators for $\im \left( \Omega^{Spin^c}_* \rightarrow \Omega^O_*\right)$ in degrees $0, 1, \dots ,n$, and tells us, for each generator, whether it corresponds to a copy of $ku_{(2)}$ or of $H\mathbb{F}_2$ under the Anderson--Brown--Peterson splitting.

In degrees $\leq 33$, we find that an element $Y_{\lambda}\in \Omega^O_*$ which is in the image of the map $\Omega^{Spin^c}_*\rightarrow \Omega^O_*$ is $2$-torsion as long as the partition $\lambda$ includes an odd number. As described in \cref{spinc cobordism intro subsection}, the $\beta$-torsion elements of $\Omega^{Spin^c}_{*}$ are exactly the $2$-torsion elements. This yields the presentation for $\Omega^{Spin^c}_*$ in degrees $\leq 33$ in the statement of the theorem. 
\end{proof}
\begin{remark}\label{remark on simplification in pf}
We marked one step in the proof of Theorem \ref{mspinc_subring in main text} with a footnote. In principle, that step in the calculation could go wrong, failing to identify all the $(Q_0,Q_1)$-torsion in $H^*(MSpin^c; \mathbb{F}_2)$, as follows: suppose there is some $\mathbb{F}_2$-linear combination of Stiefel--Whitney monomials in $H^*(MSpin^c;\mathbb{F}_2)$ which is $(Q_0,Q_1)$-torsion, but {\em none of its summands are $(Q_0,Q_1)$-torsion}. If this occurs, it would {\em not} be noticed by the method we describe, since our method only checks the $(Q_0,Q_1)$-torsion status of Stiefel--Whitney monomials. 

We handle this by a very simple idea: we make the calculation as described, and after making the full calculation, we ``check our answer'' by comparing to the known {\em additive} structure of $\Omega^{Spin^c}_*$, as follows. After running our method, we compare the rank of our calculated $(Q_0,Q_1)$-torsion in each degree to the expected rank, using the known Poincar\'{e} series for the $2$-torsion in $\pi_*(MSpin^c)$. If our method has failed to notice a {\em linear combination} of Stiefel--Whitney monomials which was $(Q_0,Q_1)$-torsion despite its summands not being $(Q_0,Q_1)$-torsion, then the rank of the $(Q_0,Q_1)$-torsion from our calculation will be too small.

We have never observed this mismatched rank to happen, i.e., it does not happen through degree $33$ in $MSpin^c_*$. If the rank mismatch were to ever occur, the fix is conceptually trivial, but computationally very hard: instead of populating $D$ with all the homogeneous $(Q_0,Q_1)$-torsion Stiefel--Whitney monomials at the start of the torsion calculation, we simply populate $D$ with the all the $(Q_0,Q_1)$-torsion Stiefel--Whitney {\em polynomials} at the start of the calculation, then re-run the calculation. This of course cannot miss any $(Q_0,Q_1)$-torsion in $H^*(MSpin^c;\mathbb{F}_2)$! Its disadvantage is simply that it is extremely computationally expensive, since the total number of homogeneous Stiefel--Whitney {\em polynomials} (not just monomials) in degrees $\leq d$ in $H^*(MSpin^c;\mathbb{F}_2)$ grows extremely quickly as $d$ grows. We carry out the calculations in the way we describe---i.e., initially populating $D$ by only the $(Q_0,Q_1)$-torsion Stiefel--Whitney {\em monomials}, rather than {\em polynomials}---to dramatically speed up calculation, and because we are able to check that the resulting answer in the end agrees with the answer we would have gotten with the much slower calculation using {\em all} the homogeneous Stiefel--Whitney polynomials.
\end{remark}

Our next result determines explicit manifolds that represent some of the bordism classes whose powers occurs as ring-theoretic generators of $\Omega^{Spin^c}_*\otimes_{\mathbb{Z}}\mathbb{F}_2$ and of $\Omega^{Spin}_*\otimes_{\mathbb{Z}}\mathbb{F}_2$ in \Cref{mspinc_image} and in \Cref{mspin_image}, respectively. The following table was calculated through degree $6$ by Thom \cite{MR0061823}. We extend the calculation through degree $17$. The symbol $D_i$ denotes the $i$-dimensional Dold manifold, defined in \cite{MR0079269}.
\begin{proposition}[Thom \cite{MR0061823}]
\label{thom_manifolds}
   Manifold representatives for elements $Y_{n}$ in Thom's partition basis for $\Omega^O_*$ are as follows:
   \begin{center}
\begin{tabular}{||c|c|}
\hline
 Element & Manifold Representative  \\
 \hline
$Y_{2}$ & $\mathbb{R}P^{2}$  \\ 
$Y_{4}$ & $\mathbb{R}P^{4} \sqcup \mathbb{R}P^2 \times \mathbb{R}P^2$ \\ 
$Y_{5}$ & $D_{5}$ \\
$Y_{6}$ & $\mathbb{R}P^6$ \\
$Y_{8}$ & $\mathbb{R}P^8\sqcup (\mathbb{R}P^4)^2 \sqcup \mathbb{R}P^4 \times (\mathbb{R}P^2)^2 \sqcup (\mathbb{R}P^2)^4$ \\
$Y_{9}$ & $D_{9} \sqcup D_{5} \times \mathbb{R}P^4 \sqcup D_{5} \times (\mathbb{R}P^2)^2$ \\
$Y_{10}$ & $\mathbb{R}P^{10} \sqcup (\mathbb{R}P^2)^5$ \\
$Y_{11}$ & $\mathbb{D}_{11} \sqcup \mathbb{D}_{9} \times \mathbb{R}P^{2}$ \\
$Y_{12}$ & $\mathbb{R}P^{12} \sqcup (\mathbb{R}P^6)^2 \sqcup \mathbb{R}P^8 \times (\mathbb{R}P^2)^2 \sqcup (D_{5}^{2} \times \mathbb{R}P^2 \sqcup (\mathbb{R}P^4)^3$ \\
$Y_{13}$ & $D_{13} \sqcup D_{11} \times \mathbb{R}P^2 \sqcup D_{9} \times \mathbb{R}P^4 \sqcup D_{8} \times \mathbb{R}P^5 \sqcup D_{5} \times \mathbb{R}P^4 \times (\mathbb{R}P^2)^2$ \\
$Y_{14}$ & $\mathbb{R}P^{14}$ \\ 
& \\ 
$Y_{16}$ & $\mathbb{R}P^{16} \sqcup \mathbb{R}P^{12}\times (\mathbb{R}P^{2})^{2} \sqcup (\mathbb{R}P^{8})^{2} \sqcup \mathbb{R}P^{8}\times (\mathbb{R}P^{4})^{2}$ \\
 & $\sqcup \mathbb{R}P^{8} \times (\mathbb{R}P^{2})^{4}\sqcup(\mathbb{R}P^{6})^{2}\times \mathbb{R}P^{4} \sqcup \mathbb{R}P^{6}\times (D_{5})^{2}$ \\
 & $ \sqcup(D_{5})^{2}\times (\mathbb{R}P^{2})^{3} \sqcup (\mathbb{R}P^{4})^{4} \sqcup (\mathbb{R}P^{4})^{2} \times (\mathbb{R}P^{2})^{4} \sqcup (\mathbb{R}P^{2})^{8}$ \\
    & \\
    $Y_{17}$ & $D_{17} \sqcup D_{13}\times \mathbb{R}P^{4} \sqcup D_{13}\times (\mathbb{R}P^{2})^{2} \sqcup \mathbb{R}P^{12}\times D_{5}\sqcup D_{11}\times \mathbb{R}P^{6} $ \\
  & $\sqcup D_{11} \times \mathbb{R}P^{4}\times \mathbb{R}P^{2} \sqcup D_{11} \times (\mathbb{R}P^{2})^{3}\sqcup D_{9}\times \mathbb{R}P^{8} \sqcup D_{9} \times \mathbb{R}P^{6} \times \mathbb{R}P^{2}$\\ 
 & $\sqcup \mathbb{R}P^{8}\times D_{5} \times (\mathbb{R}P^{2})^{2} \sqcup (\mathbb{R}P^{6})^{2}\times D_{5}\sqcup (D_{5})^{3}\times \mathbb{R}P^{2} $ \\
       & $\sqcup D_{5}\times (\mathbb{R}P^{4})^{2} \times (\mathbb{R}P^{2})^{2}\sqcup D_{5} \times \mathbb{R}P^{4} \times (\mathbb{R}P^{2})^{4} \sqcup D_{5}\times (\mathbb{R}P^{2})^{6} $\\
\hline
\end{tabular}
\end{center}
\end{proposition}
\begin{proof}
Routine calculation using Stiefel--Whitney numbers, since the unoriented cobordism class of a manifold is determined completely by its list of Stiefel--Whitney numbers. For the total Stiefel--Whitney classes of real projective spaces and of Dold manifolds, see \cite[section 4 example 4]{MR0440554} and \cite[Satz 2]{MR0079269}, respectively. Once one has the total Stiefel--Whitney class of an $n$-dimensional manifold, one pairs with the top class to get the Stiefel--Whitney numbers, one such number for each homogeneous polynomial of degree $n$ in the Stiefel--Whitney classes $w_1, \dots, w_n$. 
\end{proof}

In \cite{MR0180977}, Milnor investigates whether every spin manifold is unorientedly cobordant to the square of an orientable manifold. He shows it is true for spin manifolds of dimension $\leq 23$. The ambiguity in dimension $24$ stems from the existence of an orientable manifold whose only nonzero Stiefel--Whitney numbers are $w_4w_6w_7^2$, $w_6^4$, $w_4^6$, $w_4^3w_6^2$, and $w_4^2w_8^2$. Milnor then poses the problem of whether a spin manifold of dimension $24$ exists with these nonzero Stiefel--Whitney numbers. Anderson--Brown--Peterson stated two years later \cite{MR0219077} that, as a corollary of their main theorem, the lowest dimension in which there exists an element of $\text{Im}(\Omega^{Spin}_{*} \rightarrow \Omega^O_*)$ which is not the square of an orientable manifold is $24$ \cite{MR0219077}. In \Cref{mspin_image} we calculated $\Omega^{Spin}_{*}$ in degrees through $31$ in terms of Thom's partition basis, and in \cref{thom_manifolds} we have manifold representatives for ring-theoretic generators which suffice to generate everything in $\Omega^{Spin}_{24}$. Solving for an element of $\Omega^{Spin}_{24}$ with Milnor's prescribed Stiefel--Whitney numbers, we find an explicit manifold of the kind Milnor asked for:
\begin{theorem}\label{milnor mfld thm}
    The cobordism class $T_{24} + Y_{12}^{2}+Y_{10}^{2}Y_{2}^{2}+Y_{8}^{2}Y_{4}^{2}+Y_{8}^2Y_{2}^4+Y_{6}^{2}Y_{4}^{2}Y_{2}^{2} + Y_{5}^{4}Y_{2}^{2} +Y_4^{6} \in \text{Im}(\Omega^{Spin}_{*} \rightarrow \Omega^O_*)$ has nonzero Stiefel--Whitney numbers $w_4w_6w_7^2$, $w_6^4$, $w_4^6$, $w_4^3w_6^2$, and $w_4^2w_8^2$, and is represented by the manifold:   
    \begin{align*}
    &    (\mathbb{R}P^2)^6\times(\mathbb{R}P^6)^2 \sqcup (\mathbb{R}P^4)^6 \sqcup \mathbb{R}P^2\times(\mathbb{R}P^4)^3\times(D^5)^2 \sqcup (\mathbb{R}P^2)^2\times(\mathbb{R}P^4)^2\times(\mathbb{R}P^6)^2  \\
 &\ \ \  \sqcup (\mathbb{R}P^2)^4\times(\mathbb{R}P^8)^2 \sqcup  (\mathbb{R}P^2)^3\times\mathbb{R}P^4\times D^5\times(D^9) \sqcup (\mathbb{R}P^4)^2\times(D^5)^2\times(\mathbb{R}P^6) \\
 &\ \ \  \sqcup (\mathbb{R}P^2)^2\times D^5\times\mathbb{R}P^6\times(D^9) \sqcup (\mathbb{R}P^6)^4 \sqcup (D^5)^2\times\mathbb{R}P^6\times(\mathbb{R}P^8) \\
 &\ \ \  \sqcup (\mathbb{R}P^4)^2\times(\mathbb{R}P^8)^2 \sqcup (D^5)^3\times(D^9) \sqcup \mathbb{R}P^4\times(D^5)^2\times(\mathbb{R}P^{10}) 
 \\&\ \ \  \sqcup (\mathbb{R}P^2)^2\times(\mathbb{R}P^{10})^2 \sqcup (\mathbb{R}P^4)^2\times D^5\times(D^{11})\sqcup (\mathbb{R}P^2)^2\times D^9\times(D^{11}) 
    \\&\ \ \  \sqcup \mathbb{R}P^2\times(D^5)^2\times(\mathbb{R}P^{12}) \sqcup \mathbb{R}P^2\times\mathbb{R}P^4\times D^5\times(D^{13}) \sqcup D^5\times\mathbb{R}P^6\times(D^{13})  \\
  &\ \ \ \sqcup (D^5)^2\times(\mathbb{R}P^{14}) \sqcup (\mathbb{R}P^{12})^2 \sqcup (D^{11})\times(D^{13})\sqcup(\mathbb{R}P^2)^6\times(\mathbb{R}P^6)^2\sqcup (\mathbb{R}P^4)^6 \\
  &\ \ \ \sqcup  (\mathbb{R}P^2)^2\times(\mathbb{R}P^4)^2\times(\mathbb{R}P^6)^2 \sqcup  (\mathbb{R}P^2)^4\times(\mathbb{R}P^8)^2\sqcup (\mathbb{R}P^6)^4  \\
  &\ \ \  \sqcup (\mathbb{R}P^4)^2\times(\mathbb{R}P^8)^2 \sqcup (\mathbb{R}P^2)^2\times(\mathbb{R}P^{10})^2 \sqcup (\mathbb{R}P^{12})^2.
    \end{align*}
\end{theorem}

\section{The spin${}^c$-cobordism ring in all degrees.}

\subsection{A nonunital subring of the $2$-torsion in the spin${}^c$-cobordism ring.}

In Theorem \ref{mspinc_subring in main text}, we showed that, in degrees $\leq 33$, the cobordism classes $Y_{(5,5)},Y_{(9,9)}, Y_{(11,11)},$ and $Y_{(13,13)}\in \Omega^O_*$ lift to indecomposable $2$-torsion elements $Z_{10},Z_{18},Z_{22},$ and $Z_{26}$ in $\Omega^{Spin^c}_*$. The elements $Y_{(5,5)},Y_{(9,9)}, \dots$ are precisely those of the form $Y_{i}^2\in \Omega^O_*$ with $i$ odd and non-dyadic. It is natural to ask whether this pattern extends above degree $33$ as well. In \Cref{hassans conj}, we show that something of this kind is true, if we use the Dold manifold $D_i$, from \cite{MR0079269}, rather than the cobordism class $Y_i$ from Thom's partition basis for $\Omega^O_*$. In low degrees, the algebraic relation between the Dold manifolds and the Thom generators for $\Omega^O_*$ is as follows:
\begin{proposition}
\label{spin_manifold_reps}
The squares of odd-dimensional Dold manifolds represent the following polynomials in Thom's generators $Y_2,Y_4, Y_5, Y_6, Y_8, \dots$ for $\Omega^O_*$:
   \begin{center}
\begin{tabular}{||c|c|}
\hline
 Element & Manifold Representative  \\
 \hline
$Y_5^2$ & $D_{5}^{2}$  \\ 
$Y_9^2 + Y_5^2 Y_4^2$ & $D_{9}^{2}$ \\ 
$Y_{11}^2 + Y_9^2 Y_2^2 + Y_5^2 Y_4^2 Y_2^2$ & $D_{11}^{2}$ \\
$Y_{13}^2 + Y_{11}^2 Y_2^2 + Y_9^2 Y_4^2 + Y_8^2 Y_5^2$ & $D_{13}^{2}$ \\ \hline
\end{tabular}
\end{center}
Furthermore, each of these elements of $\Omega^O_*$ is in the image of the map $\Omega^{Spin}_*\rightarrow\Omega^O_*$. 
\end{proposition}
\begin{proof}
The manifold representatives are straightforwardly calculated from \Cref{thom_manifolds}. 
By \Cref{mspin_image}, these elements are in the image of the map $\Omega^{Spin}_*\rightarrow\Omega^O_*$. 
\end{proof}

\begin{proposition}\label{hassans conj}
For each odd integer $i$ such that $i+1$ is not a power of $2$, the Dold manifold $D_i$ has the property that $D_i\times D_i$ lifts to an indecomposable $(2,\beta)$-torsion element of $\Omega^{Spin^c}_{2i}$. It furthermore lifts to an indecomposable $2$-torsion element of $\Omega^{Spin}_{2i}$.
\end{proposition}
\begin{proof}
Dold \cite{MR0079269} proves that there exists a minimal set of generators for the cobordism ring $\Omega^O_*$ whose odd-degree elements are $D_{i}=P(2^{r}-1,s2^{r})$, where $i+1=2^{r}(2s+1)$ and $i$ is odd. In \cite[Theorem 1]{MR0180977}, Milnor proved that the map $\Omega^U_{*} \longrightarrow  \Omega^O_*$ maps onto all squares of elements in $\Omega^O_*$. This map factors through $\Omega^{Spin^c}_{*}$, so $D_{i}\times D_i$ lifts to $\Omega^{Spin^c}_{2i}$ for all non-dyadic $i$. 
By \cite[Satz 1]{MR0079269}, the mod $2$ cohomology of the Dold manifold $P(m,n)$ is given as a graded ring by
\begin{align*} 
 H^*(P(m,n);\mathbb{F}_2) &\cong \mathbb{F}_2[c,d]/(c^{m+1},d^{n+1}),
\end{align*}
with $c \in H^{1}(P(m,n); \mathbb{F}_{2})$ and $d \in H^{2}(P(m,n); \mathbb{F}_{2})$.
By \cite[Satz 2]{MR0079269}, the total Stiefel--Whitney class of $P(m,n)$ is 
\begin{align*}
w(P(m,n)) &= (1+c)^{m}(1+c+d)^{n+1} .
\end{align*}
Setting $m=2^{r}-1$ and $n=s2^{r}$, 
the total Stiefel--Whitney class of $D_{i}$ as in the statement of the theorem is then given by: 
\begin{align*}
    w(D_{i}) 
      &= (1+c)^{2^r-1}(1+c+d)^{s2^{r}},
\end{align*}
and in particular, $w_1 = 0$. Hence $D_i$ is orientable. 

To show that $D_i\times D_i$ lifts to the spin cobordism ring, one can carry out an algebraic calculation to show that none of the nonzero Stiefel--Whitney numbers of $D_i\times D_i$ are divisible by $w_2$. This is not a difficult calculation, but it is simpler to invoke the main result of Anderson's paper \cite{MR0169245}: the square of any orientable compact manifold is unorientedly cobordant to a spin manifold. Since $w_1$ vanishes on $D_i$, its square $D_i\times D_i$ must be in the image of the map $\Omega^{Spin}_{2i}\rightarrow \Omega^O_{2i}$.

It follows easily from the structure of $ko_*$, and the Anderson--Brown--Peterson splitting of $\text{MSpin}$ into a wedge of suspensions of $ko$, $ko \langle2 \rangle$, and $H\mathbb{F}_{2}$, that all elements of $\Omega^{Spin}_*$ in degrees $\equiv 2\mod 4$ are $2$-torsion. Hence, for odd $i$, any spin-cobordism class that lifts $D_i\times D_i\in \Omega^O_{2i}$ must be $2$-torsion. 

Let $\widetilde{D_i^2}$ be a lift of $D_i\times D_i\in \mathcal{N}_{2i}$ to $\Omega^{Spin}_{2i}$. The image of $\widetilde{D_i^2}$ under the map $\Omega^{Spin}_{2i}\rightarrow \Omega^{Spin^c}_{2i}$ is then a lift of $D_i^2$ to $\Omega^{Spin^c}_{2i}$, and it is $2$-torsion since it is the image of a $2$-torsion element. It is furthermore $\beta$-torsion in $\Omega^{Spin^c}_{*}$, since by the Anderson--Brown--Peterson splitting of $2$-local $MSpin^c$, every $2$-torsion element of $\Omega^{Spin^c}_{*}$ is also $\beta$-torsion.

To see that $D_{i}^2\in \Omega^O_{2i}$ lifts to an {\em indecomposable} element of $\Omega^{Spin}_*$, it is enough to observe that $\Omega^{Spin}_*/(2,\eta,\alpha,\beta)$ embeds into $\Omega^O_*$, and since $D_i$ has second Stiefel--Whitney class $d\neq 0$, $D_i$ does not lift to $\Omega^{Spin}_*$. Hence the unique lift of $D_i^2$ to $\Omega^{Spin}_*/(2,\eta,\alpha,\beta)$ is indecomposable, hence any lift of $D_i^2$ to $\Omega^{Spin}_*$ is indecomposable. A completely analogous argument establishes that any lift of $D_i^2$ to $\Omega^{Spin^c}_*$ is also indecomposable.
 \end{proof}

Since Dold \cite{MR0079269} showed that $D_i\in\Omega^O_i$ can be written as Thom's generator $Y_i$ plus decomposables in the same degree, Proposition \ref{hassans conj} tells us that {\em some of} the patterns exhibited in table \eqref{bahri-gilkey table 2} are not limited to degrees $\leq 33$, and indeed extend to {\em all} degrees. Namely, for odd non-dyadic $i$, if we write $Z_{2i}$ for a lift of $D_i\times D_i \in\Omega^O_{2i}$ to an indecomposable $2$-torsion element of $\Omega^{Spin^c}_{2i}$ (guaranteed to exist by Proposition \ref{hassans conj}), and for even $i$ we write $Z_{2i}$ for a lift of $Y_i^2\in \Omega^O_{2i}$ to an element of $\Omega^{Spin^c}_{2i}$, then we have:
\begin{theorem}\label{large nonunital subring thm}
Consider the spin${}^c$-cobordism ring as a graded algebra over the graded ring $S:= \mathbb{Z}_{(2)}[\beta, Z_{2j} : j\geq 2,\ \mbox{j\ non-dyadic}]/(\beta Z_{2j},\ 2Z_{2j}\mbox{\ for\ odd\ } j).$
Then the ideal $\left(Z_{2j} : j\ \mbox{odd}\right)$ of $S$ embeds, as a non-unital graded $S$-algebra, into the $2$-torsion ideal $\pi_*(Z)$ of the spin${}^c$-cobordism ring.
\end{theorem}

\subsection{Mod $2$ spin${}^c$-cobordism, up to uniform $F$-isomorphism.}
\label{fisomorphism}

There is a subring of the mod $2$ spin${}^c$-cobordism ring $\Omega^{Spin^c}_*\otimes_{\mathbb{Z}}\mathbb{F}_2$ which is generated by 
\begin{itemize}
\item the large nonunital subring of the $2$-torsion in $\Omega^{Spin^c}_*$ constructed in Theorem \ref{large nonunital subring thm},
\item and Stong's generators $y_4, y_8, y_{12},y_{16},\dots$ of the mod-torsion mod-$2$ spin${}^c$-cobordism ring $(\Omega^{Spin^c}_*/tors) \otimes_{\mathbb{Z}}\mathbb{F}_2$. 
\end{itemize}
This subring of $\Omega^{Spin^c}_*\otimes_{\mathbb{Z}}\mathbb{F}_2$
 is strictly smaller than $\Omega^{Spin^c}_*\otimes_{\mathbb{Z}}\mathbb{F}_2$ itself. However, we will now show that this subring is {\em uniformly $F$-isomorphic} to $\Omega^{Spin^c}_*\otimes_{\mathbb{Z}}\mathbb{F}_2$. See Definition \ref{def of f-iso} for the definition of a uniform $F$-isomorphism. 

\begin{theorem}\label{f-iso thm}
The mod $2$ spin${}^c$-cobordism ring is uniformly $F$-isomorphic to the graded $\mathbb{F}_{2}$-algebra
\begin{equation}\label{b pres}
 \mathbb{F}_{2}\left[\beta,y_{4i}, Z_{4j-2} : i\geq 1,\ j\geq 1,\ 2j\mbox{\ not\ a\ power\ of\ }2\right]/(\beta Z_{4j-2}),
\end{equation}
with $\beta$ the Bott element in degree $2$, with $y_{4i}$ in degree $4i$, and with $Z_{4j-2}$ in degree $4j-2$.
\end{theorem}
\begin{proof}
Write $tors$ for the ideal of $\Omega^{Spin^c}_*$ consisting of the $2$-torsion elements.
Stong \cite[Proposition 14]{MR192516} proved that $\left(\Omega^{Spin^c}_*/tors\right)\otimes_{\mathbb{Z}}\mathbb{F}_2$ is a polynomial $\mathbb{F}_2$-algebra on generators $y_2$ and $y_4,y_8,y_{12}, \dots$. Since $\Omega^{Spin^c}_2\cong \mathbb{Z}$ generated by $\beta = [\mathbb{C}P^1]$, Stong's generator $y_2$ agrees modulo $2$ with the Bott element $\beta$. Hence $\left(\Omega^{Spin^c}_*/tors\right)\otimes_{\mathbb{Z}}\mathbb{F}_2$ is isomorphic as a graded $\mathbb{F}_2$-algebra to $\mathbb{F}_{2}\left[\beta,y_{4i} : i\geq 1\right]$.

Now let $B$ denote the graded subring of $\Omega^{Spin^c}_*\otimes_{\mathbb{Z}}\mathbb{F}_2$ generated by $\beta$, by $y_{4i}$ for all $i>1$, and by the mod $2$ reductions of the $(2,\beta)$-torsion elements in $\Omega^{Spin^c}_*$ from Theorem \ref{hassans conj} which lift 
$D_{i}\times D_{i}$ for all odd non-dyadic $i$. Since $B$ contains all the squares of elements in $\Omega^{Spin^c}_*\otimes_{\mathbb{Z}}\mathbb{F}_2$, the graded $\mathbb{F}_2$-algebra map
\begin{align*} B &\hookrightarrow \Omega^{Spin^c}_*\otimes_{\mathbb{Z}}\mathbb{F}_2\end{align*}
is a uniform $F$-isomorphism.

Let $\tilde{T}$ denote the kernel of the ring map $\Omega^{Spin^c}_*\otimes_{\mathbb{Z}}\mathbb{F}_2 \rightarrow \left( \Omega^{Spin^c}_*/tors\right)\otimes_{\mathbb{Z}}\mathbb{F}_2$. 
Filter $B$ by powers of the ideal $B\cap \tilde{T}$, i.e., equip $B$ with the $(B\cap \tilde{T})$-adic filtration. By the Anderson--Brown--Peterson splitting and by Theorem \ref{hassans conj}, the associated graded ring $E_0B$ is isomorphic to $\mathbb{F}_2[\beta]$ tensored with the image of $B$ in $\Omega^O_*$ and reduced modulo the relations $\beta \cdot x = 0$ for all $x\in B\cap \tilde{T}$, i.e., $E_0B$ is isomorphic to \eqref{b pres}. 

We claim that $B$ itself is isomorphic to \eqref{b pres}. The $(B\cap \tilde{T})$-adic filtration on $B$ is additively split, so $B \cong E_0B$ as graded $\mathbb{F}_2$-vector spaces. In principle, the ring structure on $E_0B$ could differ from the ring structure on $B$ if the multiplication on $B$ were to exhibit $(B\cap \tilde{T})$-adic filtration jumps, i.e., when we multiply two elements $x,y$ of $B$, with $x$ of $B\cap\tilde{T}$-adic filtration $i$ and with $y$ of $(B\cap \tilde{T})$-adic filtration $j$, we could perhaps get an element of $(B\cap \tilde{T})$-adic filtration $>i+j$. 

However, even if filtration jumps occur, $E_0B$ is still isomorphic to the graded $\mathbb{F}_2$-algebra with presentation \eqref{b pres}. This is by a freeness argument similar to the classical argument that, if the associated graded of a filtered commutative $k$-algebra is a polynomial (i.e., free commutative) $k$-algebra, then the original filtered commutative $k$-algebra must also have been free commutative. The argument is as follows. Let $\mathcal{C}$ be the category of pairs $(A,S)$, where $A$ is a graded-commutative $\mathbb{F}_2[\beta]$-algebra, and $S$ is a set of homogeneous elements of $A$ such that $\beta\cdot x = 0$ for all $x\in S$. There is a forgetful functor from $\mathcal{C}$ to the category $\Subsets$ of pairs $(S_0,S_1)$ in which $S_0,S_1$ are sets and $S_1\subseteq S_0$. The forgetful functor $\mathcal{C}\rightarrow \Subsets$ sends $(A,S)$ to the underlying sets of $A$ and of $S$. The graded $\mathbb{F}_2[\beta]$-algebra \eqref{b pres} is the free object of $\mathcal{C}$ on the pair 
\[ \left( \{ y_4,y_8,y_{12},y_{16},\dots\} \cup \{ Z_2, Z_6, Z_{10}, Z_{18}, Z_{22}, \dots\} ,\ \{ Z_2, Z_6, Z_{10}, Z_{18}, Z_{22}, \dots\} \right).\]
By Proposition \ref{hassans conj}, the elements $\{ Z_2, Z_6, Z_{10}, Z_{18}, Z_{22}, \dots\}$ are $\beta$-torsion in $B$, not merely in $E_0B$. Hence there are no relations on $E_0B$ except those which make it an object of the category of $\mathcal{C}$, and $B$ lives in $\mathcal{C}$ as well, i.e., $E_0B$ and $B$ are isomorphic in the category $\mathcal{C}$. Hence $E_0B$ and $B$ are isomorphic as graded $\mathbb{F}_2$-algebras, and hence $\Omega^{Spin^c}_*\otimes_{\mathbb{Z}}\mathbb{F}_2$ is uniformly $F$-isomorphic to the $\mathbb{F}_2$-algebra \eqref{b pres}, as claimed.
\end{proof}

In \cref{results}, we showed that $\Omega^{Spin^c}_*\otimes_{\mathbb{Z}}\mathbb{F}_2$ is not isomorphic to a polynomial algebra. Nevertheless, since an $F$-isomorphism induces a homeomorphism on the prime spectra (see \cite[Proposition B.8]{MR0298694} or \cite[Lemma 29.46.9]{stacks-project}), we have:
\begin{corollary}\label{f-iso cor}
The topological space $\Spec \left( \Omega^{Spin^c}_*/(2,\beta)\right)$ is homeomorphic to $\Spec$ of a polynomial $\mathbb{F}_2$-algebra on countably infinity many generators.

Furthermore, the topological space $\Spec \left(\Omega^{Spin^c}_*\otimes_{\mathbb{Z}}\mathbb{F}_2\right)$ is homeomorphic to $\Spec$ of the ring \eqref{b pres}.
\end{corollary}

\end{document}